\documentclass[letterpaper]{article}
\usepackage{color}
\usepackage{framed}
\usepackage[small]{titlesec}
\usepackage{theorem,amsmath,amssymb,amscd, verbatim}
\usepackage[all]{xy}

\usepackage{booktabs}
\usepackage[hyperfootnotes=false, colorlinks, linkcolor={blue}, citecolor={magenta}, filecolor={blue}, urlcolor={blue}]{hyperref}

\theoremstyle{change}
\allowdisplaybreaks
\renewcommand{\a}{{\mathfrak a}}
\renewcommand{\H}{\mathbb H}
\newcommand{\A}{{\mathbb A}}
\newcommand{\Q}{{\mathbb Q}}
\newcommand{\Z}{{\mathbb Z}}
\newcommand{\R}{{\mathbb R}}
\newcommand{\B}{{\mathcal B}}
\newcommand{\C}{{\mathbb C}}
\newcommand{\f}{{\mathrm f}}
\newcommand{\bs}{\backslash}

\newcommand{\p}{\mathfrak p}
\newcommand{\OF}{{\mathfrak o}}
\newcommand{\Ad}{{\rm Ad}}
\newcommand{\GL}{{\rm GL}}
\newcommand{\PGL}{{\rm PGL}}
\newcommand{\SL}{{\rm SL}}
\newcommand{\SO}{{\rm SO}}

\newcommand{\GSp}{{\rm GSp}}
\newcommand{\Sp}{{\rm Sp}}
\newcommand{\diag}{\mathrm{diag}}
\newcommand{\PGSp}{{\rm PGSp}}

\newcommand{\St}{{\rm St}}
\newcommand{\triv}{1}

\newcommand{\Tr}{\mathrm{Tr}}

\newcommand{\Mat}{{\rm M}}

\newcommand{\trace}{{\rm tr}}

\newcommand{\vl}{{\rm vol}}
\newcommand{\disc}{{\rm disc}}

\newcommand{\AI}{{\mathcal{AI}}}

\newcommand{\T}[1]{{}^t\!{{#1}}}
\DeclareMathOperator{\Cl}{Cl}

\newcommand{\mat}[4]{{\setlength{\arraycolsep}{0.5mm}\left[
\begin{smallmatrix}#1&#2\\#3&#4\end{smallmatrix}\right]}}
\newcommand{\qed}{\hspace*{\fill}\rule{1ex}{1ex}}
\newcommand{\forget}[1]{}

\def\qdots{\mathinner{\mkern1mu\raise0pt\vbox{\kern7pt\hbox{.}}\mkern2mu
\raise3.4pt\hbox{.}\mkern2mu\raise7pt\hbox{.}\mkern1mu}}

\newenvironment{proof}{\vspace{1ex}\noindent\emph{Proof.}\hspace{0.5em}}
	{\hfill\qed\vspace{2ex}}

\newtheorem{lemma}{Lemma.}[section]
\newtheorem{theorem}[lemma]{Theorem.}

\newtheorem{corollary}[lemma]{Corollary.}
\newtheorem{proposition}[lemma]{Proposition.}

\newtheorem{remark}[lemma]{Remark.}

\newtheorem{conjecture}[lemma]{Conjecture.}

\begin{document}

\title{Explicit refinements of B\"ocherer's conjecture for Siegel modular forms of squarefree level}

\author{Martin Dickson$^1$, Ameya Pitale$^2$, Abhishek Saha$^3$, and Ralf Schmidt$^2$}
\date{%
    $^1$King's College London\\%
    $^2$University of Oklahoma\\
    $^2$Queen Mary University of London\\
    [2ex]%
    \today
}

\maketitle

\begin{abstract}We formulate an explicit refinement of B\"ocherer's conjecture for Siegel modular forms of degree 2 and squarefree level, relating weighted averages of Fourier coefficients with special values of $L$-functions. To achieve this, we compute the relevant local integrals that appear in the refined global Gan-Gross-Prasad conjecture for Bessel periods as proposed by Yifeng Liu. We note several consequences of our conjecture to arithmetic and analytic properties of $L$-functions and Fourier coefficients of Siegel modular forms.
\end{abstract}

\tableofcontents

\section{Introduction}
\subsection{B\"ocherer's conjecture}

Let $f$ be a Siegel cusp form\footnote{For definitions and background on Siegel modular forms, see~\cite{Klingen1990}.} of degree 2 and weight $k$ for the group $\Sp(4,\Z)$. Then $f$ has a Fourier expansion $$f(Z)
=\sum_{S } a(f, S) e^{2 \pi i \trace(SZ)},$$ where the Fourier coefficients $a(f,S)$ are indexed by matrices $S$ of the form
\begin{equation}\label{e:matrixform}
 S=\mat{a}{b/2}{b/2}{c},\qquad a,b,c\in\Z, \qquad a>0, \qquad \disc(S) := b^2 - 4ac < 0.
 \end{equation}
For two matrices $S_1$, $S_2$ as above, we write $S_1 \sim S_2$ if there exists $A \in \SL(2,\Z)$ such that $S_1 = {}^t\!A S_2A$.  Using the defining relation for Siegel cusp forms, we see that \begin{equation}\label{siegelinv}
a(f, S_1) = a(f,S_2)\qquad\text{if }S_1\sim S_2,
\end{equation}
thus showing that $a(f, S)$ only depends on the $\SL(2,\Z)$-equivalence class of the matrix $S$, or equivalently, only on the proper equivalence class of the associated binary quadratic form.

Let $d < 0$ be a fundamental discriminant\footnote{Recall that an integer $n$ is a fundamental discriminant if \emph{either} $n$ is a squarefree integer congruent to 1 modulo 4 \emph{or} $n = 4m$ where $m$ is a squarefree integer congruent to 2 or 3 modulo 4.}. Put $K = \Q(\sqrt{d})$ and let $\Cl_K$ denote the ideal class group of $K$. It is well-known that the $\SL(2,\Z)$-equivalence classes of binary quadratic forms of discriminant $d$ are in natural bijective correspondence with the elements of $\Cl_K$. In view of the comments above, it follows that the notation $a(f, c)$ makes sense for $c \in \Cl_K$. We define

\begin{equation}\label{rfddef}
R(f, K) = \sum_{c \in \Cl_K}a(f, c) = \sum_{\substack{S / \sim \\ \disc(S)=d}} a(F, S).
\end{equation}

For odd $k$, it is easy to see that $R(f, K)$ equals 0. If $k$ is even, B\"ocherer \cite{boch-conj} made a remarkable conjecture that relates the quantity $R(f, K)$  to the central value of a certain $L$-function.

\begin{conjecture}[B\"ocherer~\cite{boch-conj}]\label{bochconj} Let $k$ be even, $f \in S_k(\Sp(4,\Z))$ be a non-zero Hecke eigenform and $\pi_f$ be the associated automorphic representation of $\GSp(4,\A)$. Then there exists a constant $c_f$ depending only on $f$ such that for any imaginary quadratic field $K=\Q(\sqrt{d})$ with $d<0$ a fundamental discriminant, we have

 $$ |R(f, K)|^2 = c_f \cdot w(K)^{2} \cdot |d|^{k-1} \cdot L_\f(1/2, \pi_f \otimes \chi_{d}) .$$

 Above, $\chi_{d} = \left(\frac{d}{ \cdot }\right)$ is the Kronecker symbol (i.e., the quadratic Hecke character associated via class field theory to the field $\Q(\sqrt{d})$), $w(K)$ denotes the number of distinct roots of unity inside $K$, and $L_\f(s, \pi_f \otimes \chi_{d})$ denotes the (finite part of the) associated degree $4$ Langlands $L$-function.\footnote{All global $L$-functions in this paper are normalized to have the functional equation taking $s \mapsto 1-s$. We denote the completed $L$-functions by $L(s,\ )$ and reserve $L_\f(s,\ )$ for the finite part of the $L$-function, i.e., without the gamma factor $L_\infty(s,\ )$. Thus $L(s,\ ) = L_\f(s,\ )L_\infty(s,\ )$.}
\end{conjecture}

As far as we know, B\"ocherer did not speculate what the constant $c_f$ is exactly (except for the case when $f$ is a Saito-Kurokawa lift). For many applications, both arithmetic and analytic, it would be of interest to have a conjectural formula where the constant $c_f$ is known precisely, and one of the original motivations of this paper was to do exactly that.

This paper deals with refinements of B\"ocherer's conjecture for Siegel newforms with squarefree level. A special case of the results of this paper is the following \emph{precise formula} for $c_f$, valid for all eigenforms $f$ of full level and weight $k>2$ that are not Saito-Kurokawa lifts: \begin{equation}\label{e:bochconst}c_f = \frac{2^{4k-4} \cdot \pi^{2k+1}}{(2k-2)!}\cdot \frac{L_\f(1/2, \pi_f)}{L_\f(1, \pi_f, \Ad)} \cdot \langle f, f \rangle.
\end{equation}
Above, $\langle f, f \rangle$ is the Petersson norm of $f$ normalized as in \eqref{eqn:petersson-def}, $L_\f(s, \pi_f, \Ad)$ is the finite part of the adjoint (degree 10) $L$-function for $\pi_f$, and $L_\f(s, \pi_f)$ is the finite part of the spinor (degree 4) $L$-function for $\pi_f$. Indeed, combining the results of this paper with recent work of Furusawa and Morimoto \cite{FM}, we get the following theorem.

\begin{theorem}[c.f. Theorem 2 of \cite{FM}]  Let $k$ be even, $f \in S_k(\Sp(4,\Z))$ be a non-zero Hecke eigenform that is not a Saito-Kurokawa lift and $\pi_f$ be the associated automorphic representation of $\GSp(4,\A)$. Then for any imaginary quadratic field $K=\Q(\sqrt{d})$ with $d<0$ a fundamental discriminant, we have

 $$ \frac{|R(f, K)|^2}{ \langle f, f \rangle} = \frac{2^{4k-4} \ \pi^{2k+1}}{(2k-2)!} \ w(K)^2 \ |d|^{k-1}\frac{L_\f(1/2, \pi_f)}{L_\f(1, \pi_f, \Ad)}{ \ L_\f(1/2, \pi_f \otimes \chi_{d})}.$$
\end{theorem}
A similar theorem is true for Saito-Kurokawa lifts and was proved in B\"ocherer in \cite{boch-conj}.

Conjecture \ref{bochconj} can be extended naturally as follows. For any character $\Lambda$ of the finite group $\Cl_K$, we make the definition
\begin{equation}\label{rfdeflam}
R(f, K, \Lambda) = \sum_{c \in \Cl_K}a(f, c) \Lambda^{-1}(c).
\end{equation}
Let $\AI(\Lambda^{-1})$  be the automorphic representation of $\GL(2,\A)$ given by the automorphic induction of $\Lambda^{-1}$ from $\A_K^\times$; it is generated by (the adelization of)
the theta-series of weight 1 and character $\chi_d$ given by $$\theta_{\Lambda^{-1}}(z) = \sum_{0 \ne \a \subset O_K}\Lambda^{-1}(\a) e^{2 \pi i N(\a)z}.$$
 We make the following refined version of Conjecture \ref{bochconj} (a further generalization of which to the case of  squarefree levels is stated as Theorem \ref{t:mainclassi} later in this introduction).

\begin{conjecture}\label{weakgenboch} Let $f \in S_k(\Sp(4,\Z))$ be a non-zero Hecke eigenform of weight $k\ge 2$ and $\pi_f$ the associated automorphic representation of $\GSp(4,\A)$. Suppose that $f$ is not a Saito-Kurokawa lift. Then for any imaginary quadratic field $K=\Q(\sqrt{d})$ with $d<0$ a fundamental discriminant and any character $\Lambda$ of $\Cl_K$, we have
  \begin{align*}\frac{|R(f, K, \Lambda)|^2}{\langle f, f \rangle} &=  2^{2k-4} \ w(K)^2 \ |d|^{k-1}\frac{L(1/2, \pi_f \times \AI(\Lambda^{-1}))}{ \ L(1, \pi_f, \Ad)} \\ &=  \frac{2^{4k-4} \ \pi^{2k+1}}{(2k-2)!} \ w(K)^2 \ |d|^{k-1}\frac{L_\f(1/2, \pi_f \times \AI(\Lambda^{-1}))}{ \ L_\f(1, \pi_f, \Ad)}\end{align*}
\end{conjecture}

\begin{remark}
For Saito-Kurokawa lifts, a similar equality is actually a theorem; see Theorem \ref{t:mainskformula}.
\end{remark}

\begin{remark}
We note that $L_\f(1, \pi_f, \Ad)$ is not zero by Theorem 5.2.1 of \cite{transfer}.
\end{remark}

\begin{remark}\label{rem:archfactors}
 The values of the relevant archimedean $L$-factors used above are as follows,
 \begin{align*}
  L_\infty(1/2, \pi_f\times \AI(\Lambda^{-1}))&=2^4 (2\pi)^{-2k} (\Gamma(k-1))^2,\\
  L_\infty(1, \pi_f, \Ad)&=2^5 (2 \pi)^{-4k-1}(\Gamma(k-1))^2\,\Gamma(2k-1).
 \end{align*}
\end{remark}

\begin{remark}
 If $\Lambda=1$, then $ R(f, K, 1) = R(f, K)$ and $$L(1/2, \pi_f \otimes \AI(1)) =L(1/2, \pi_f) L(1/2, \pi_f \otimes \chi_{d}).$$ Hence, Conjecture~\ref{bochconj} is a special case of Conjecture~\ref{weakgenboch} with the particular value of $c_f$ given in \eqref{e:bochconst}.
\end{remark}

\begin{remark}
Note that Conjecture~\ref{weakgenboch} implies that whenever $ R(f, K, \Lambda) \neq 0$, one also has $L(1/2, \pi_f \times \AI(\Lambda^{-1})) \neq 0.$ This is a special case of the Gan-Gross-Prasad conjectures \cite{ggp}.

\end{remark}

As we explain in the next section, Conjecture~\ref{weakgenboch} follows from the refined Gan-Gross-Prasad (henceforth GGP) conjectures as stated by Yifeng Liu~\cite{yifengliu}. These refined conjectures involve certain local factors which are essentially integrals of local matrix coefficients. The primary goal of the present paper is to compute these local factors exactly in several cases. Our calculation of the local factor at infinity for all  $k>2$ and at all places where $\pi_f$ is unramified allows us to make Conjecture~\ref{weakgenboch} (though we do our archimedean calculations only for $k>2$, we speculate that the same answer holds for $k=2$). We also calculate these local factors for certain Iwahori-spherical representations. As a consequence, we are able to generalize Conjecture~\ref{weakgenboch} to all newforms of squarefree level. For the definition of the local factors and an overview of our results, we refer the reader to Sect.~\ref{intro:local}; for the resulting explicit refinement of B\"ocherer's conjecture, see Sect.~\ref{s:classicalrefined} and in particular Theorem \ref{t:mainclassi}. The details and proofs of these appear in Sects.~\ref{s:localint} and \ref{s:global} respectively. We also prove some consequences of these refined conjectures in Sect.~\ref{s:conseq} (some of these consequences are summarized later in this introduction in Section \ref{s:introconseq}).
\subsection{Bessel periods and the refined Gan-Gross-Prasad conjecture}
We now explain how Conjectures \ref{bochconj} and \ref{weakgenboch} fit into the framework of the GGP conjectures made in \cite{ggp}. Let $F$ be a number field and $H \subseteq G$ be reductive groups over $F$ such that $Z_H(\A_F)H(F)\bs H(\A_F)$ has finite volume (where $Z_H$ denotes the center of $H$). Let $\pi$ (resp.~$\sigma)$  be an irreducible, unitary, cuspidal, automorphic representation of $G(\A_F)$ (resp.~$H(\A_F)$). Suppose that the central characters $\omega_\pi$, $\omega_\sigma$ of $\pi$ and $\sigma$ are inverses of each other. Then if $\phi_\pi$ and $\phi_\sigma$ are automorphic forms in the space of $\pi$ and $\sigma$ respectively, one can form the global period
$$
 P(\phi_\pi, \phi_\sigma) = \int\limits_{Z_H(\A_F)H(F)\bs H(\A_F)} \phi_\pi(h) \phi_\sigma(h)\,dh.
$$

The basic philosophy is that, in certain cases, the quantity $|P(\phi_\pi, \phi_\sigma)|^2$ should be essentially equal to the value $L(1/2, \pi \otimes \sigma)$. The earliest prototype of this is due to Waldspurger \cite{Waldspurger1985}, who dealt with the case $H=\SO(2)$, $G=\SO(2,1) \simeq \PGL(2)$. The case $H=\SO(2,1)$, $G=\SO(2,2)$ is the famous triple product formula (note that $\SL(2) \times \SL(2)$ is a double cover for $\SO(2,2)$), developed and refined by various people over the years, including Harris-Kudla \cite{harris-kudla-1991}, Watson \cite{watson-2008}, Ichino \cite{ichino}, and Nelson-Pitale-Saha \cite{nelson-pitale-saha}. Gross and Prasad made a series of conjectures in this vein for orthogonal groups; this was extended to other classical groups by Gan-Gross-Prasad \cite{ggp}. In their original form, the Gross-Prasad and GGP  conjectures do not predict an exact formula for $|P(\phi_\pi, \phi_\sigma)|^2$ but instead deal with the conditions for its non-vanishing in terms of the non-vanishing of $L(1/2, \pi \otimes \sigma)$. However, a recent work of Yifeng Liu \cite{yifengliu},
extending that of Ichino-Ikeda \cite{ichino-ikeda} and Prasad-Takloo-Bighash~\cite{PrasadTaklooBighash2011}, makes a refined GGP conjecture by giving a precise conjectural formula for $|P(\phi_\pi, \phi_\sigma)|^2$ for a wide family of automorphic representations.

The case of Liu's conjecture that interests us in this paper is $G =\PGSp(4) \simeq \SO(3,2)$, and $H=T_SN$ equal to the \emph{Bessel subgroup}, which is an enlargement of $T_S \simeq \SO(2)$ with a
unipotent subgroup $N$.

Let $S$ be a symmetric $2 \times 2$ matrix that is anisotropic over $F$; in particular this implies that $d=\disc(S)$ is not a square in $F$. Define  a subgroup $T_S$ of $\GL(2)$ by
\begin{equation}\label{defTS}
T_S(F) = \{g \in \GL(2,F):\; \T{g}Sg =\det(g)S\}.
\end{equation}
It is easy to verify that $T_S(F) \simeq K^\times$ where $K = F(\sqrt{d})$. We  consider $T_S$ as a subgroup of $\GSp(4)$ via
\begin{equation}\label{embedding}
T_S \ni g \longmapsto
\left[\begin{smallmatrix}
g & 0\\
0 & \det(g)\, \T{g^{-1}}
\end{smallmatrix}\right] \in \GSp(4).
\end{equation}
Let us denote by $N$ the subgroup of $\GSp(4)$ defined by
$$
N = \{n(X) =\mat{1}{X}{0}{1}\;|\;\T{X} = X\}.
$$
 Fix a non-trivial additive character $\psi$ of $F \bs \A_F$.  We define the
character $\theta_S$ on $N(\A)$ by $\theta_S(n(X))=
\psi(\trace(SX))$. Let $\Lambda$ be a character of $T_S(F) \bs T_S(\A_F) \simeq  K^\times \bs \A_K^\times$ such that $\Lambda |_{\A_F^\times}= 1$. Let $(\pi,V_\pi)$ be a cuspidal, automorphic representation of $\GSp(4,\A)$ with trivial central character. For an automorphic form $\phi \in V_\pi$, we define the global Bessel period $B(\phi, \Lambda)$ on $\GSp(4,\A)$ by
\begin{equation}\label{defbessel}
  B(\phi, \Lambda) =
  \int\limits_{\A_F^\times T_S(F)\bs T_S(\A_F)}\;\int\limits_{N(F) \bs N(\A_F)}\phi(tn)\Lambda^{-1}(t) \theta_S^{-1}(n)\,dn\,dt.
\end{equation}
The GGP conjecture in this case  can be stated as follows.
\begin{conjecture}[Special case of global GGP]\label{weakggp}
 Let $\pi$, $\Lambda$ be as above. Suppose that for almost all places $v$ of $F$, the local representation $\pi_v$ is generic. Suppose also that there exists an automorphic form $\phi$ in the space of $\pi$ such that $B(\phi, \Lambda) \neq 0$. Then $L(1/2, \pi \times \AI(\Lambda^{-1})) \neq 0.$
\end{conjecture}

\begin{remark}The global GGP \cite{ggp} actually predicts more in this case, namely that $L(1/2, \pi \times \AI(\Lambda^{-1}))$ $\neq 0$ if and only if there exists a representation $\pi'$ in the Vogan packet of $\pi$ and an automorphic form $\phi'$ in the space of $\pi'$ such that  $B(\phi', \Lambda) \neq 0$.
\end{remark}

\begin{remark}\label{bphiclas}In the special case $F=\Q$, if we take  $\phi$ to be the adelization of a weight $k$ Siegel cusp form,  $\psi$ the standard additive character,  $d<0$ a fundamental discriminant, and $\Lambda$ corresponding to a character of $\Cl_K$, then up to a constant, one has the relation (see~\cite{fur}) $$ B(\phi, \Lambda) = R(f, K, \Lambda).$$
\end{remark}
We now state Liu's refinement of Conjecture \ref{weakggp} in our case. For simplicity, we restrict to the case $F= \Q$ and write $\A$ for $\A_\Q$. Let the matrix $S$, the group $T=T_S$, the field $K=\Q(\sqrt{d})$, and the character $\chi_d$ be as above. Let $\pi = \otimes_v \pi_v$ be an irreducible, unitary, cuspidal, automorphic representation of $\GSp(4,\A)$ with trivial central character and $\Lambda$ be a unitary Hecke character of $K^\times \bs \A_K^\times$ such that $\Lambda|_{\A^\times} = 1$.  Fix the local Haar measures $dn_v$ on $N(\Q_v)$ (resp.\ $dt_v$ on $T(\Q_v)$) to be the standard additive measure at all places (resp.\ such that the maximal compact subgroup has volume one at almost all places). Fix the global measures $dn$ and $dt$ to be the global \emph{Tamagawa measures}.  Note that $dn = \prod_v dn_v$. Define the constant $C_T$ by $dt= C_T \prod_v dt_v$.  For each automorphic form $\phi$ in the space of $\pi$, define the global Bessel period $B(\phi, \Lambda)$ by~\eqref{defbessel} and define the Petersson norm $\langle \phi, \phi\rangle = \int_{Z(\A)G(\Q) \bs G(\A)} |\phi(g)|^2\,dg$ where we use the Tamagawa measure on $Z(\A)G(\Q) \bs G(\A)$. For each place $v$, fix a $\GSp(4,\Q_v)$-invariant Hermitian inner product $\langle\, , \rangle_v$ on $\pi_v$. For $\phi_v \in V_{\pi_v}$, define
$$
 J_v(\phi_v ) =\frac{L(1, \pi_v, \Ad)L(1, \chi_{d,v})\int_{\Q_v^\times \bs T(\Q_v)}\int_{N(\Q_v)}\frac{\langle \pi_v(t_vn_v) \phi_v , \phi_v \rangle}{\langle \phi_v , \phi_v\rangle} \Lambda_v^{-1}(t_v) \theta_S^{-1}(n_v)\,dn_v\,dt_v}{\zeta_{\Q_v}(2)\zeta_{\Q_v}(4)L(1/2, \pi_v \otimes \AI(\Lambda_v^{-1}))}.
$$
Strictly speaking, the integral above may not converge absolutely, in which case one defines it via regularization (see~\cite[p.~6]{yifengliu}). It can be shown that $J_v(\phi_v )=1$ for almost all places. We now state the refined conjecture as phrased by Liu.

\begin{conjecture}[Yifeng Liu~\cite{yifengliu}]\label{c:liu}
 Let $\pi$, $\Lambda$ be as above. Suppose that for almost all places $v$ of $\Q$ the local representation $\pi_v$ is generic. Let $\phi=\otimes_v \phi_v$ be an automorphic form in the space of $\pi$. Then
 \begin{equation}\label{e:refggp}
  \frac{|B(\phi, \Lambda)|^2}{\langle \phi, \phi \rangle} = \frac{C_T}{S_\pi}\frac{\zeta_\Q(2)\zeta_\Q(4)L(1/2, \pi \times \AI(\Lambda^{-1}))}{L(1, \pi, \Ad)L(1, \chi_{d})} \prod_v J_v(\phi_v ),
 \end{equation}
 where $\zeta_\Q(s) = \pi^{-s/2}\Gamma(s/2)\zeta(s)$ denotes the completed Riemann zeta function and $S_\pi$ denotes a certain integral power of 2, related to the Arthur parameter of $\pi$. In particular,
 $$
  S_\pi = \begin{cases} 4 &\text{ if } L(s, \pi) = L(s, \pi_1)L(s, \pi_2) \text{ for irreducible cuspidal representations } \pi_i \text{ of } \GL(2,\A), \\ 2 &\text{ if } L(s, \pi) = L(s, \Pi) \text{ for an irreducible cuspidal representation } \Pi \text{ of } \GL(4,\A). \end{cases}
 $$
\end{conjecture}
Liu gave a proof of the above conjecture in the special case of Yoshida lifts. The proof in the case of the non-endoscopic analogue of Yoshida lifts has been very recently given by Corbett~\cite{corbett}. Recently, in a remarkable work \cite{FM}, Furusawa and Morimoto have proved Conjecture \ref{c:liu} for $\Lambda=1$ under some mild hypotheses.
\subsection{Computing local integrals}\label{intro:local}
We now briefly describe our main local result. Fix the following purely local data (where for convenience we have dropped all subscripts).
\begin{enumerate}
 \item A non-archimedean local field $F$ of characteristic 0. We let $\OF$ denote the ring of integers of $F$ and $q$ the cardinality of the residue class field.
 \item A symmetric invertible matrix $S=\mat{a}{b/2}{b/2}{c}$. Put $d = b^2-4ac$. Denote $K = F(\sqrt{d})$ if $d$ is a non-square in $F$ and $K = F \times F$ if $d$ is a square in $F$. Assume that $a,b \in \OF$, $c \in \OF^\times$, and $d$ generates the discriminant ideal of $K/F$; if $K = F \times F$, the discriminant ideal is just $\OF$. Let $\chi_{K/F}$ denote the quadratic character on $F^\times$ associated to the extension $K/F$.
 \item Haar measures $dn$ and $dt$ on $N(F)$ and $F^\times \bs T_S(F)$ respectively, normalized so that the subgroups $N(\OF)$ and $\OF^\times \bs T(F) \cap \GL(2,\OF)$ have volume $1$. Henceforth, we denote $T_S(\OF) = T(F) \cap \GL(2,\OF)$.
 \item An irreducible, admissible, unitary representation $(\pi,V_\pi)$ of $\GSp(4,F)$ with trivial central character\footnote{Our results immediately extend to the slightly more general case of unramified central character by considering a suitable unramified twist of $\pi$.} and invariant inner product $\langle\,,\,\rangle$, and a choice of vector $\phi \in V_\pi$.
 \item  A character $\Lambda$ of $K^\times$, which we also think of as a character of $T_S(F)$, that satisfies  $\Lambda|_{F^\times} = 1$.
 \item An additive character $\psi$ of $F$ with conductor $\OF$. Define a character $\theta_S$ on $N(F)$ by $\theta(n(X))=\psi(\trace(SX))$.
\end{enumerate}
Given all the above data, we define the local quantity
\begin{equation}\label{e:defi}
 J(\phi, \Lambda) :=\frac{L(1, \pi, \Ad)L(1, \chi_{K/F})\int_{F^\times \bs T_S(F)}\int_{N(F)}\frac{\langle \pi(tn) \phi , \phi \rangle}{\langle \phi , \phi\rangle} \Lambda^{-1}(t) \theta_S^{-1}(n)\,dn\,dt }{\zeta_{F}(2)\zeta_{F}(4)L(1/2, \pi \times \AI(\Lambda^{-1}))},
\end{equation}
where $\zeta_F(s)=(1-q^{-s})^{-1}$. Note that the above quantity is exactly what appears in Liu's refined GGP conjecture \eqref{e:refggp} above. If all the data above is \emph{unramified} (see Section~\ref{s:localintbeg}), then one can show that $J(\phi, \Lambda)=1.$ Note also that $J(\phi, \Lambda)$ does not change if $\phi$ is multiplied by a constant.

In this paper, we explicitly evaluate the local quantity $J(\phi, \Lambda)$ for two types of ramified cases. First, we evaluate it in the case when $\pi$ is unramified but the extension $K/F$ is a ramified field extension. This is achieved in Theorem \ref{t:localunram}. Secondly, we consider the case when $\pi$ itself is (possibly) ramified, but has a vector fixed by the  congruence subgroup $P_1$ of $\GSp(4,\OF)$ defined by
\begin{equation}\label{P1DEF}P_1 = \left\{ A \in \GSp(4,\OF) : A \in \left[\begin{smallmatrix}
\OF & \OF & \OF & \OF\\
\OF & \OF & \OF & \OF\\
\p & \p & \OF & \OF\\
\p & \p & \OF  & \OF\\
\end{smallmatrix}\right] \right\}.\end{equation}

Using the standard classification (see \cite{sch} or \cite{PS1} or the tables later in this paper) this implies that $\pi$ is one of the types I, IIa, IIIa, Vb, Vc, VIa or VIb. For each $\pi$ as above, we choose a certain orthogonal basis $\B$ of the $P_1$-fixed subspace of $\pi$. Then for each vector $\phi$ in $\B$, we \emph{exactly compute} the local integrals $J(\phi, \Lambda)$, under the assumption that  $\Lambda$ is unramified, $F$ has odd residual characteristic and $K/F$ is an unramified field extension. For the complete results, see Theorem \ref{t:localintmain}, which is the technical heart of this paper. Its proof relies on computations of various local integrals and matrix coefficients, which are performed in Section \ref{s:localint}.
\subsection{A refined conjecture in the classical language}\label{s:classicalrefined}
We now use our results described in Section \ref{intro:local} to translate Conjecture~\ref{c:liu} to the classical language of Fourier coefficients for Siegel modular forms of squarefree level.

For any congruence subgroup $\Gamma$ of $\Sp(4,\Z)$, and any integer $k$, let $S_k(\Gamma)$ denote the space of Siegel cusp forms of weight $k$ with respect to $\Gamma$. For any positive integer $N$, let the congruence subgroup $\Gamma_0(N)$ be the usual congruence subgroup,
\begin{equation}\label{defu1n}
 \Gamma_0(N) = \Sp(4,\Z) \cap \left[\begin{smallmatrix}\Z& \Z&\Z&\Z\\\Z& \Z&\Z&\Z\\N\Z& N\Z&\Z&\Z\\N\Z&N \Z&\Z&\Z\\\end{smallmatrix}\right].
\end{equation}
\emph{Now, suppose that $N$ is squarefree.} By a newform of weight $k$ for
 $\Gamma_0(N)$, we mean an element $f \in  S_k(\Gamma_0(N))$ with the
following properties.
\begin{enumerate}
\item \label{itemfa} $f$ lies in the orthogonal complement of the space of oldforms $S_k(\Gamma_0(N))^{\text{old}},$ as defined in \cite{sch}.
\item \label{itemfb}$f$ is an eigenform for the local Hecke algebras for all primes $p$ not dividing $N$ and an eigenfunction of the $U(p)$ operator (see \cite{sahaschmidt}) for all primes dividing $N$.
\item \label{itemfc}The adelization of $f$ generates an irreducible cuspidal representation $\pi_f$ of $\GSp(4,\A)$.
\end{enumerate}
It is known that any newform $f$ defined as above is of one of two types:
\begin{enumerate}
\item \textbf{CAP}: A CAP newform $f$ is one for which the associated automorphic representation $\pi_f = \otimes_v \pi_{f,v}$ is nearly equivalent to a constituent of a global induced representation of a proper parabolic subgroup of $\GSp(4,\A)$. These newforms $f$ do not satisfy the Ramanujan conjecture. Furthermore, if $k \ge 3$, the CAP newforms are  exactly those that are obtained via Saito-Kurokawa lifting from classical newforms of weight $2k-2$ and level $N$ \cite{schsaito}. For $k=1, 2$, one can potentially also have CAP newforms that are not of Saito-Kurokawa type; these are associated to the Borel or Klingen parabolics. Finally, if $f$ is CAP, the representations $\pi_{f,v}$ are non-generic for all places $v$.

\item \textbf{Non-CAP}: These are the newforms $f$ in the complement of the above set. Weissauer proved \cite{weissram} that the non-CAP newforms always satisfy the Ramanujan conjecture at all primes $p \nmid N$ whenever $k \ge 3$. Another key property of the non-CAP newforms is that the local representations $\pi_{f,p}$ are generic for all $p \nmid N$. In terms of the standard classification (see Section 2.2 of \cite{NF}), the  representations $\pi_{f,p}$ are type I if $p \nmid N$, and one of types IIa, IIIa, Vb/c, VIa, VIb if $p|N$. (Conjecturally, they are not of type Vb/c, since these are non-tempered representations.)
\end{enumerate}

Given any newform $f$ in $S_k(\Gamma_0(N))$, any fundamental discriminant $d<0$, and any character $\Lambda$ of the ideal class group of $K=\Q(\sqrt{d})$, we can define the quantity $R(f,K,\Lambda)$ exactly as in \eqref{rfdeflam}.
Then our local results lead to the following theorem.

\begin{theorem}\label{t:mainclassi}
 Let $N$ be squarefree and $f$ be a non-CAP newform in $S_k(\Gamma_0(N))$. For $d<0$ a fundamental discriminant, and $\Lambda$ a character of $\Cl_K$, define the quantity $R(f,K,\Lambda)$ as in \eqref{rfdeflam}. Let $\pi_f = \otimes_v \pi_{f,v}$ be the automorphic representation of $\GSp(4,\A)$ associated to $f$. Suppose that $N$ is odd, $k>2$, and $\big( \frac{d}{p} \big) = -1$ for all primes $p$ dividing $N$. Then, assuming the truth of Conjecture \ref{c:liu}, we have
 \begin{equation}\label{bochrefinedeq}\frac{|R(f, K, \Lambda)|^2}{\langle f, f \rangle} =  2^{2k-s} \ w(K)^2 \ |d|^{k-1}\frac{L(1/2, \pi_f \times \AI(\Lambda^{-1}))}{ \ L(1, \pi_f, \Ad)} \prod_{p|N} J_p,
 \end{equation}
 where $s = 5$ if $f$ is a weak Yoshida lift\footnote{In this case, one can show that a weak Yoshida lift is automatically a Yoshida lift; see Corollary 4.17 of \cite{sahapet}.} in the sense of \cite{sahapet} and $s= 4$ otherwise.
  The quantities $J_p$ for $p|N$ are given as follows:
  $$
   J_p = \begin{cases}(1+p^{-2})(1+p^{-1}) & \text{ if }   \pi_{f,p} \text{ is of type {\rm IIIa}}, \\ 2(1+p^{-2})(1+p^{-1}) & \text{ if }  \pi_{f,p}\text{ is of type {\rm VIb}}, \\ 0 & \text{ otherwise.} \end{cases}
  $$
\end{theorem}

\begin{remark} The assumptions $N$ odd and $\big(\frac{d}{p}\big) = -1$ arise due to the assumptions on our local non-archimedean computations.  The assumption $k>2$ is only needed because we use the formula for the matrix coefficients of holomorphic discrete series representations from \cite[Proposition A.1]{knightlyli2016}. It seems very plausible that this formula also holds for limit of discrete series representations, which would allow us to easily extend Theorem  \ref{t:mainclassi} to the case $k=2$. \end{remark}

\begin{remark}For CAP newforms of Saito-Kurokawa type, we prove a similar theorem to the above that does not rely on the truth of any conjecture. For the precise statement, see Theorem \ref{t:mainskformula}.\end{remark}

As mentioned earlier, Conjecture \ref{c:liu} is now known in a few cases, and therefore the statement of Theorem \ref{t:mainclassi} is unconditional in those cases. For ease of reference, we state this separately next.

\begin{theorem}\label{t:summaryintro}
 Let $N$ be squarefree and $f$ be a non-CAP newform in $S_k(\Gamma_0(N))$. For $d<0$ a fundamental discriminant, and $\Lambda$ a character of $\Cl_K$, define the quantity $R(f,K,\Lambda)$ as in \eqref{rfdeflam}. Let $\pi_f = \otimes_v \pi_{f,v}$ be the automorphic representation of $\GSp(4,\A)$ associated to $f$. Suppose that $N$ is odd, $k>2$, and $\big( \frac{d}{p} \big) = -1$ for all primes $p$ dividing $N$. Suppose further that at least one of the following hold: a) $\Lambda=1$ b) $f$ is a Yoshida lift. Then  we have
 \begin{equation}\frac{|R(f, K, \Lambda)|^2}{\langle f, f \rangle} =  2^{2k-s} \ w(K)^2 \ |d|^{k-1}\frac{L(1/2, \pi_f \times \AI(\Lambda^{-1}))}{ \ L(1, \pi_f, \Ad)} \prod_{p|N} J_p,
 \end{equation}
 where $s = 5$ if $f$ is a Yoshida lift and $s= 4$ otherwise.
  The quantities $J_p$ for $p|N$ are as in Theorem \ref{t:mainclassi}.
\end{theorem}
\subsection{Some consequences of the refined conjecture}\label{s:introconseq}
The formula \eqref{bochrefinedeq}  is the promised explicit refinement of B\"ocherer's conjecture for Siegel newforms of squarefree level. Our main global result, Theorem \ref{t:main}, is a mild extension of this formula that includes the case of oldforms. Historically, exact formulas relating periods and $L$-functions (such as Watson's triple product formula \cite{watson-2008} or Kohnen's refinement of Waldspurger's formula \cite{kohnwalds}) have had various applications, ranging from quantum unique ergodicity in level aspect \cite{nelson-pitale-saha} to non-vanishing of central values modulo primes \cite{bruinnonvan}. So one can expect
Theorem \ref{t:main} to have many interesting consequences as well. Towards the end of this paper (see Section \ref{s:conseq}) we work out some of these consequences. We  note a few of them below.

\medskip

\textbf{Averages of $L$-functions.} By combining Theorem \ref{t:main} with some results proved in \cite{kst2} and \cite{dickson}, we obtain some quantitative asymptotic formulas for averages of $L$-functions. For example, if $\B_k^T$ is an orthogonal Hecke basis for the space of Siegel cusp forms of full level and weight $k>2$ that are not Saito-Kurokawa lifts, then we show that Conjecture \ref{c:liu} implies:

  \begin{equation}\label{introaverage}\sum_{f\in \B_k^T} \frac{L_\f(1/2, \pi_f \times \AI(\Lambda^{-1}))}{ \ L_\f(1, \pi_f, \Ad) \ L_\f(1, \chi_d)} - \frac{k^3}{l \pi^6} \ll k^{7/3},\end{equation}
where $l = 1$ if $\Lambda^2=1$ and $l=2$ otherwise. In fact, we prove a version of \eqref{introaverage} for forms of squarefree level; see Theorem \ref{t:aver} and Corollary \ref{corr:aver}.

\medskip

\textbf{Bounds for Fourier coefficients.} Theorem \ref{t:main} combined with the GRH allows us to predict the following best possible upper bound for the size of $R(f,K, \Lambda)$ for all non-CAP newforms $f$ in $S_k(\Gamma_0(N))$,
\begin{equation}\label{introfourierbd}|R(f,K, \Lambda)| \ll_{\epsilon} \langle f, f \rangle^{1/2} (2\pi e)^{k} \ k^{-k+\frac{3}{4}} |d|^{\frac{k-1}{2}} (Nkd)^\epsilon.\end{equation}

\medskip
\textbf{Arithmeticity of $L$-values.} The refined B\"ocherer conjecture  implies that for all $f$, $\Lambda$ as in Theorem \ref{t:mainclassi} such that $f$ has algebraic integer Fourier coefficients,  the quantity \begin{equation}\label{alglvalues}\pi^{2k+1} \cdot \langle f, f\rangle \cdot \frac{L_\f(1/2, \pi_f \times \AI(\Lambda^{-1}))}{ \ L_\f(1, \pi_f, \Ad)}\end{equation} is, up to an exactly specified rational number, an \emph{algebraic integer} (see Proposition \ref{prop:algebraicity}). We find this interesting because even the algebraicity of \eqref{alglvalues} appears to not have been conjectured previously.

\medskip

The above results are of course conditional on Conjecture \ref{c:liu}. However, there are important cases where Conjecture \ref{c:liu} is already proven.
As mentioned earlier,  Furusawa and Morimoto have recently proved this conjecture \cite{FM} for trivial ideal class character under some mild
hypotheses. Furthermore, Furusawa and Morimoto write in \cite[Remark 7]{FM} ``\emph{In the sequel, we shall pursue [...] extension to the case when the character of the ideal class group is not necessarily
trivial.}"  The conjecture is also known in full generality for Yoshida lifts, see \cite{yifengliu}, and for the non-endoscopic analogue of Yoshida lifts, see \cite{corbett}. The statement of Theorem \ref{t:summaryintro} above summarizes the types of Siegel newforms of squarefree levels where we can make unconditional statements at present.

Applying our formula to Yoshida lifts, we note the following unconditional \emph{integrality} result about families of $L$-values of \emph{classical modular forms}, which may be of independent interest:
\begin{proposition}\label{simularith}Let $p$ be an odd prime and let $g_1, g_2$ be distinct classical primitive forms of level $p$ and weights $\ell$, $2$ respectively, where $\ell >2$, $\ell$ even. Assume that the Atkin-Lehner eigenvalues of $g_1$ and $g_2$ at $p$ are equal. Then there exists a positive real number $\Omega$ (depending only on $g_1$, $g_2$) such that for any imaginary quadratic field $K$ with the property that $p$ is inert in $K$, and any character $\Lambda$ of $\Cl_K$,

$$ \ \Omega \times \disc(K)^{\ell/2} \times L_\f(1/2, \pi_{g_1} \times \AI(\Lambda^{-1})) \ L_\f(1/2, \pi_{g_2}\times \AI(\Lambda^{-1}))$$ is a totally-positive algebraic integer in the field generated by $\Lambda$-values and the coefficients of $g_1$, $g_2$.
\end{proposition}

\subsection{Acknowledgements}We thank Atsushi Ichino for pointing out a gap in an earlier version of this paper. We thank Paul Nelson and Kaj B\"auerle for noticing an error in our formula for the matrix coefficients in an earlier version of this paper. We thank Shunsuke Yamana for noticing an error in our calculations for the powers of 2 in our formulas in an earlier version of this paper, and for forwarding us their preprint \cite{hsieh-yamana}. We thank Masaaki Furusawa and Kazuki Morimoto for communicating to us an issue with some of the constants in an earlier version of this paper and for helpful discussions about their preprint \cite{FM}. We are grateful to Kazuki Morimoto for allowing us to use his calculation for $S=1$ in Section \ref{s:arch}.
A.S. acknowledges the support of the EPSRC grant EP/L025515/1.
\subsection{Notations}\label{notsec}
Throughout this paper, the letter $G$  will denote the algebraic group $\GSp(4)$, defined as
$$
\GSp(4,R)=\{g\in\GL(4,R):\:\T{g}Jg=\lambda(g)J,\:\lambda(g)\in R^\times\},\qquad J=\left[\begin{smallmatrix}&&1\\&&&1\\-1\\&-1\end{smallmatrix}\right]
$$
for a commutative ring $R$ with identity. The subgroup for which $\lambda(g)=1$ is denoted $\Sp(4, R)$. We denote by $N$ the unipotent radical of the Siegel parabolic subgroup of $\GSp(4,R)$. Explicitly,
\begin{equation}\label{Ndefeq}
 N=\{\left[\begin{smallmatrix}
         1&&x&y\\&1&y&z\\&&1\\&&&1
        \end{smallmatrix}\right]:\;x,y,z\in R\}.
\end{equation}

The symbol $\A$ will denote the ring of adeles over $\Q$. For an additive group $R$, ${\rm Sym}(2,R)$ denotes the group of size 2 symmetric matrices over $R$. All (local and global) $L$-functions and $\varepsilon$-factors are defined via the local Langlands correspondence. The global $L$-functions  denoted by $L(s, \ )$  include the archimedean factors. The finite part of an $L$-function (without the archimedean factors) is denoted by $L_\f(s, \ )$.
\section{The local integral in the non-archimedean case}\label{s:localint}

\subsection{Assumptions and basic facts}\label{s:localintbeg}
The following notations will be used throughout Section \ref{s:localint}. Let $F$ be a non-archimedean local field of characteristic zero, $\OF$ the ring of integers of $F$, and $\p$ the maximal ideal of $\OF$. Let $\varpi\in\OF$ be a fixed generator of $\p$, and $q=\#(\OF/\p)$ be the cardinality of the residue class field. The normalized valuation on $F$ is denoted by $v$, and the normalized absolute value by $|\cdot|$. The character of $F^\times$ obtained by restricting $|\cdot|$ to $F^\times$ is denoted by $\nu$. We fix a non-trivial character $\psi$ of $F$ with conductor $\OF$ once and for all.

Table \eqref{Tablelocal} lists all the irreducible, admissible, non-supercuspidal representations of $G(F)$ supported in the minimal parabolic subgroup.  For further explanation of the notation, we refer to Sect.~2.2 of \cite{NF} (also see the comments in Section \ref{matcoeffsec} of this paper).

\begin{equation}\label{Tablelocal}\renewcommand{\arraystretch}{.7}
 \begin{array}{rlccccc}
  \toprule
   \text{type}&&\text{representation}&\,{\rm g}&\,P_1&\,K\\
  \toprule
   {\rm I}&& \chi_1 \times \chi_2 \rtimes \sigma\ \mathrm{(irreducible)}&\bullet&4&1\\
  \midrule
   {\rm II}&{\rm a}&\chi \St_{\GL(2)}\rtimes\sigma&\bullet&1&0\\
  \cmidrule{2-6}
   &{\rm b}&\chi \triv_{\GL(2)}\rtimes\sigma&&3&1\\
  \midrule
  {\rm III}&{\rm a}&\chi \rtimes \sigma \St_{\GSp(2)}&\bullet&2&0\\
  \cmidrule{2-6}
   &{\rm b}&\chi \rtimes \sigma \triv_{\GSp(2)}&&2&1\\
  \midrule
   {\rm IV}&{\rm a}&\sigma\St_{\GSp(4)}&\bullet&0&0\\
  \cmidrule{2-6}
   &{\rm b}&L(\nu^2,\nu^{-1}\sigma\St_{\GSp(2)})&&2&0\\
  \cmidrule{2-6}
   &{\rm c}&L(\nu^{3/2}\St_{\GL(2)},\nu^{-3/2}\sigma)&&1&0\\
  \cmidrule{2-6}
   &{\rm d}&\sigma\triv_{\GSp(4)}&&1&1\\
  \midrule
   {\rm V}&{\rm a}&\delta([\xi,\nu \xi], \nu^{-1/2} \sigma)&\bullet&0&0\\
  \cmidrule{2-6}
  &{\rm b}&L(\nu^{1/2}\xi\St_{\GL(2)},\nu^{-1/2} \sigma)&&1&0\\
  \cmidrule{2-6}
   &{\rm c}&L(\nu^{1/2}\xi\St_{\GL(2)},\xi\nu^{-1/2}\sigma)&&1&0\\
  \cmidrule{2-6}
   &{\rm d}&L(\nu\xi,\xi\rtimes\nu^{-1/2}\sigma)&&2&1\\
  \midrule
   {\rm VI}&{\rm a}&\tau(S,\nu^{-1/2}\sigma)&\bullet&1&0\\
  \cmidrule{2-6}
  &{\rm b}&\tau(T,\nu^{-1/2}\sigma)&&1&0\\
  \cmidrule{2-6}
   &{\rm c}&L(\nu^{1/2}\St_{\GL(2)},\nu^{-1/2}\sigma)&&0&0\\
  \cmidrule{2-6}
   &{\rm d}&L(\nu,1_{F^\times}\rtimes\nu^{-1/2}\sigma)&&2&1\\
  \bottomrule
 \end{array}
\end{equation}

The ``g'' column indicates which representations are generic. The last two columns $P_1$ and $K$ give the dimensions of the subspace of vectors fixed by the Siegel congruence subgroups $P_1$ (see \eqref{congruencesubgroupseq}) and the maximal compact subgroup $K= G(\OF)$ respectively, assuming that the inducing characters are unramified. This data will only be useful in Section \ref{s:localint}. The representations IVb, IVc are never unitary and hence are not relevant for our global applications. The representations that contain a $K$-fixed vector are known as spherical representations; we see from the table that these are of types I, IIb, IIIb, IVd, Vd, or VId.

We introduce the following notations regarding Bessel models for $\GSp_4$. For $a,b,c\in F$ let
\begin{equation}\label{Sdefeq}
 S = \mat{a}{b/2}{b/2}{c}.
\end{equation}
Assume that $d = b^2-4ac$ is not zero and put $D=d/4= -\det(S)$. If $D \notin F^{\times 2}$, then let $\Delta = \sqrt{D}$ be a square root of $D$ in an algebraic closure $\bar F$, and  $K=F(\Delta)$. If $D \in F^{\times 2}$, then let $\sqrt{D}$ be a square root of $D$ in $F^\times$,  $K = F\times F$, and  $\Delta = (-\sqrt{D},\sqrt{D}) \in K$. In both cases $K$ is a two-dimensional $F$-algebra containing $F$, $K=F+F \Delta$, and $\Delta^2=D$. Let
\begin{equation}\label{AFdefeq}
 A=A_S=\{\mat{x+yb/2}{yc}{-ya}{x-yb/2}:\:x,y \in F \}.
\end{equation}
Then $A$ is a two-dimensional $F$-algebra naturally containing $F$. One can verify that
\begin{equation}\label{AFdefeq2}
 A=\{g\in \Mat_2(F):\:^tgSg=\det(g)S\}.
\end{equation}
We write $T=T_S = A^\times$. We define an isomorphism of $F$-algebras,
\begin{equation}\label{ALisoeq}
 A \stackrel{\sim}{\longrightarrow} K, \qquad \mat{x+yb/2}{yc}{-ya}{x-yb/2}\longmapsto x + y \Delta.
\end{equation}
The restriction of this isomorphism to $T$ is an isomorphism $T \stackrel{\sim}{\longrightarrow} K^\times$, and  we identify characters of $T$ and characters of $K^\times$ via this isomorphism. We embed $T$ into $G(F)$ via the map
\begin{equation}\label{Tembeddingeq}
 t\longmapsto \mat{t}{}{}{\det(t)t'},\qquad t\in T, \quad t'= {}^tt^{-1}.
\end{equation}
For $S$ as in \eqref{Sdefeq}, we define a character $\theta=\theta_{a,b,c}=\theta_S$ of $N$ by
\begin{equation}\label{thetaSsetupeq}
 \theta(\left[\begin{smallmatrix} 1 &&x&y \\ &1&y&z \\ &&1& \\ &&&1 \end{smallmatrix}\right]) = \psi(ax+by+cz) = \psi (\trace(S\mat{x}{y}{y}{z}))
\end{equation}
for $x,y,z \in F$. Every character of $N$ is of this form for uniquely determined $a,b,c$ in $F$, or, alternatively, for a uniquely determined symmetric $2\times2$ matrix $S$. It is easily verified that $\theta(tnt^{-1})=\theta(n)$ for $n\in N$ and $t\in T$. We refer to $T \ltimes N$ as the \emph{Bessel subgroup} defined by the character $\theta$ (or the matrix $S$). Given a character $\Lambda$ of $T$ (identified with a character of $K^\times$ as explained above), we can define a character $\Lambda\otimes\theta$ of $T \ltimes N$ by
$$
 (\Lambda\otimes\theta)(tn)=\Lambda(t)\theta(n)\qquad\text{for $n\in N$ and $t\in T$}.
$$
Every character of $T \ltimes N$ whose restriction to $N$ coincides with $\theta$ is of this form for an appropriate $\Lambda$.

Let $(\pi,V)$ be an admissible representation of $G(F)$. Let $S$ and $\theta$ be as above, and let $\Lambda$ be a character of the associated group $T$. By a \emph{$(\Lambda,\theta)$-Bessel functional} for $\pi$, we mean a non-zero functional $\beta : V \rightarrow \C$ such that $\beta(\pi(tn)v) = \Lambda(t)\theta(n) \beta(v)$, for all $t \in T, n \in N$ and $v \in V$. We note that if $\pi$ admits a central character $\omega_\pi$ and a $(\Lambda,\theta)$-Bessel functional, then $\Lambda|_{F^\times}=\omega_\pi$.

We now make the following assumptions, which will be in force through all of Section \ref{s:localint}.
\begin{enumerate}
 \item $S= \mat{a}{b/2}{b/2}{c}$ with $a,b \in \OF$, $c \in \OF^\times$, and $d=b^2-4ac \neq 0$.
 \item  The element $d$ has the property that it generates the discriminant ideal of $K/F$, where  $K:=F(\sqrt{d})$ or $K:=F \times F$ depending on whether $d \notin F^{\times 2}$ or $d \in F^{\times 2}$. In particular, $d$ is a unit if and only if either a) $K=F \times F$ or b) $K$ is the unramified quadratic field extension of $F$.
 \item The additive character $\psi$ of $F$ has conductor $\OF$ and the character $\Lambda$ on $T(F) \cong K^\times$ satisfies $\Lambda|_{F^\times} = 1$.
 \item    $(\pi,V)$ is a unitary, irreducible, admissible representation of $G(F)$ with trivial central character and invariant hermitian inner product $\langle\cdot,\cdot\rangle$.
 \item   The relevant Haar measures are normalized so that the groups $N(\OF)$ and $\OF^\times \bs T_S(\OF)$ both have volume $1$.
\end{enumerate}
As explained in \cite[Section 2.3]{PS1}, there is no loss of generality in making these assumptions.
For any vector $\phi \in V$, define $\Phi_\phi$ to be the function on $G(F)$ defined by \begin{equation}\label{e:defmc}\Phi_\phi(g) = \langle \pi(g)\phi , \phi \rangle / \langle \phi, \phi \rangle.\end{equation}
For each $t \in T = T_S \simeq K^\times$, consider the integrals
\begin{equation}\label{J0kdefeqbeg}
 J_0^{(k)}(\phi, t):=\int\limits_{N(\p^{-k})}\Phi_\phi(nt)\theta_S^{-1}(n)\,dn,
\end{equation}
where $k$ is a positive integer. It was proved by Liu \cite{yifengliu} that for each $t$ the values $J_0^{(k)}(\phi, t)$ stabilize as $k \rightarrow \infty$, and hence the definition
\begin{equation}\label{J0phiteq}
 J_0(\phi,t) = \lim_{k \rightarrow \infty} J_0^{(k)}(\phi,t)
\end{equation}
is meaningful. In fact,  Liu showed there is some open compact subgroup $U$ of $N$ such that $$\int\limits_{U}\Phi_\phi(nt)\theta_S^{-1}(n)\,dn = \int\limits_{U'}\Phi_\phi(nt)\theta_S^{-1}(n)\,dn,$$ for \emph{every} open compact subgroup $U'$ containing $U$. Define the quantity
\begin{equation}\label{e:defj0}
 J_0(\phi, \Lambda) = \int\limits_{F^\times \bs T} J_0(\phi, t)\Lambda^{-1}(t)\,dt
\end{equation} whenever the integral on the right side above converges absolutely
and define the normalized quantity
\begin{equation}\label{e:defj}
 J(\phi, \Lambda)=\frac{L(1, \pi, \Ad)L(1, \chi_{K/F}) }{\zeta_{F}(2)\zeta_{F}(4)L(1/2, \pi \times \AI(\Lambda^{-1}))} \ J_0(\phi, \Lambda)
\end{equation}
whenever $L(1, \pi, \Ad) \neq \infty$ (this happens whenever $\pi$ is a generic $L$-packet).

A slightly different local integral  (relying on a different regularization) was defined by Qiu in \cite{qiu}, which is useful for dealing with non-tempered representations, including ones for which Liu's integral does not converge, or for which $L(s, \pi, \Ad)$ has a pole at $s=1$. We now describe this.

Let $\mathrm{Sym}(2,F)$ be the set of symmetric matrices of size 2 over $F$, and $\mathrm{Sym}_0(2,F) = \mathrm{Sym}(2,F) \cap \GL(2,F)$. For each $t \in T$, Qiu shows \cite[Prop.\ 2]{qiu} that there exists an integrable, locally constant function $W_{\phi, t}$ on $\mathrm{Sym}_0(2,F)$ with the property that
\begin{equation}\label{qiunormalize}
 \int\limits_{N(F)}\:\int\limits_{\mathrm{Sym}(2,F)} \Phi_\phi(nt) \theta_Y^{-1}(n) f(Y)\,dY\,dn = \int\limits_{\mathrm{Sym}_0(2,F)} W_{\phi, t}(Y) f(Y)\,dY
 \end{equation}
for all Schwartz-Bruhat (locally constant, compactly supported) functions $f$ on $\mathrm{Sym}(2,F)$. Observe here that the left side of the integral is well-defined because the Fourier transform of a Schwartz-Bruhat function is also Schwartz-Bruhat; observe also that the function $W_{\phi,t}$ (once we know that it exists) is clearly unique. Furthermore, Qiu defines a height function $\Delta$ on the group $T$, which has the property of being equal to the constant function 1 whenever $K/F$ is a field. Now, following Qiu, let \begin{equation}\label{e:defj0qiu}
 J_0^*(\phi,t) : = W_{\phi, t}(S), \quad J_0^*(s, \phi) = \int\limits_{F^\times \bs T} J_0^*(\phi, t)\Delta(t)^s\,dt,
\end{equation}
and finally
\begin{equation}\label{e:defjqiu}
 J^*(\phi)= \frac{L(s+1, \pi, \Ad)L(1, \chi_{K/F}) }{\zeta_{F}(2)\zeta_{F}(4)L(s+1/2, \pi) L(s+1/2, \pi \times \chi_{K/F}))} \ J_0^*(s,\phi) \bigg |_{s=0}.
\end{equation}
Note that in the integrals $J^*(\phi)$, $J_0^*(s, \phi)$, there is no presence of $\Lambda$ (they correspond to $\Lambda=1$).
\subsection{The local integral in the spherical case}
We remind the reader that the assumptions  stated at the beginning of \S\ref{s:localintbeg} continue to be in force for this and all the remaining local sections. In this subsection, we will prove the following theorem.
\begin{theorem}\label{t:localunram}
 Suppose that $\pi$ has a (unique up to multiples) $G(\OF)$-fixed vector and let $\phi$ be a non-zero such vector. If the residue field characteristic $q$ is even, assume moreover that $F= \Q_2$.
 \begin{enumerate}
  \item If $\pi$ is a type I representation and $\Lambda$ is an unramified character, then $J(\phi, \Lambda) = 1$.
  \item  If $\pi$ is a type IIb representation  of the form $\chi1_{\GL(2)}\rtimes\chi^{-1}$, then $J^*(\phi) = 2$.
 \end{enumerate}
\end{theorem}

The above theorem is already known when $K=F \times F$ or if $K/F$ is an \emph{unramified} field extension. Indeed, the first assertion of the theorem in these cases is due to Liu \cite{yifengliu} while the second assertion is due to Qiu \cite{qiu}.\footnote{Note that the conditions for the evaluation of the unramified local integrals $J(\phi, \Lambda)$, $J^*(\phi)$ in these cases, set out explicitly in \cite[Sect.\ 2.1]{yifengliu}, are met in our case due to the above assumptions on $S$, $\pi$, $\Lambda$. $\psi$, $\phi$.} So we will only need to take care of the case when $K/F$ is a ramified field extension. Therefore, throughout this subsection, we make the following assumption:

\medskip
\textbf{Assumption:} \emph{The vector $\phi \in V_\pi$ is $G(\OF)$-fixed. $K/F$ is a ramified field extension (equivalently, $d \notin F^{\times2}$ and $d \notin \OF^\times$). Moreover, if the residue field characteristic $q$ is even, then $F= \Q_2$. }

\medskip
The above assumption and the fact that $d$ generates the discriminant ideal of $K/F$ implies that $b/2 \in \OF$, and moreover $d/4 \in \OF^\times \cup \varpi \OF^\times$ with $d/4 \in \OF^\times$ possible only if $q=2$. Moreover it is clear (using the well-known description of discriminants of quadratic extensions of $\Q_2$) that if $d/4 \in \OF^\times$, then $1-d/4$ is in $\varpi \OF^\times$ and therefore is a uniformizer in $F$.

Next, observe that the matrix $S' = \mat{-\frac{d}{4c}}{}{}{c}$  satisfies the first assumption of \S\ref{s:localintbeg} and moreover, the matrix $A=\left[\begin{smallmatrix}1\\x&1\end{smallmatrix}\right]$ with $x=-c^{-1}b/2$ has the property that $S'=\,^t\!ASA.$ As $A \in \GL(2,\OF)$, a simple computation shows that $J_0^{(k)}(\phi, t) = \int\limits_{N(\p^{-k})}\Phi_\phi(nt')\theta_{S'}^{-1}(n)\,dn$, where $t' = A^{-1}tA \in T_{S'}$. Hence in the sequel we can (and will) replace $S$ by $S'$, \emph{and hence assume for the rest of this subsection that $b=0$.}
\begin{lemma}
 Let $J_0(\phi, t)$, $J_0^*(\phi, t)$ be as defined in the previous subsection. Then $J_0(\phi, t) = J_0^*(\phi, t)$.
\end{lemma}
\begin{proof}
Let $k$ be a positive integer, and put $f=f_k$ equal to $\vl(\p^k\mathrm{Sym}(2,\OF) )^{-1}$ times the characteristic function of $S + \p^k\mathrm{Sym}(2,\OF)$. Then, as $k\rightarrow \infty$, the right hand side of \eqref{qiunormalize} approaches  $ W_{\phi, t}(S) = J_0^*(\phi, t)$. On the other hand, the left hand side is equal to
$$
 \int\limits_{N(F)} \Phi_\phi(nt)  \psi(-\trace(S n))\bigg(\:\int\limits_{\mathrm{Sym}(2,\OF)} \psi(-\trace(\varpi^k Y n)) dY\bigg)\,dn.
$$
Using the fact that $\psi$ has conductor $\OF$, we see that the inner integral equals $0$ unless $n = \mat{n_1}{n_2/2}{n_2/2}{n_3}$ with $n_1, n_2, n_3 \in \varpi^{-k}\OF$, in which case the inner integral equals 1. It follows that the left hand side equals $\int_{n \in \mat{\p^{-k}}{\p^{-k}/2}{\p^{-k}/2}{\p^{-k}}}\Phi_\phi(nt)\psi(-\trace(S n))\,dn$, which equals $ J_0(\phi, t)$ for all $k$ sufficiently large. The proof is complete.
\end{proof}

\begin{remark} The content of the above lemma is that for the part of the local integral that is over the unipotent subgroup $N$,  the regularization employed by Liu and the regularization using the Fourier transform (as implemented by Qiu) are equivalent. This is also alluded to in \cite[Remark 3.17]{yifengliu}.
\end{remark}

\begin{lemma}\label{J0*rel} Suppose that  $\Lambda=1$. Then the integrals defining $J_0(\phi, \Lambda)$ and $J_0^*(0, \phi)$ are absolutely convergent, and $J_0(\phi, \Lambda) = J_0^*(0, \phi)$.
\end{lemma}
\begin{proof}This is immediate from the previous lemma, as the hypothesis $d \notin F^{\times2}$ implies that $F^\times \bs T$ is a compact set.
\end{proof}

Let $t_K \in T$ be an element such that the element $\varpi_K$ of $K^\times$ corresponding to $t_K$ (under the fixed isomorphism $T \cong K^\times$) is a uniformizer for $\OF_K$. It follows that the entries of $t_K$ belong to $\OF$ and $\det(t_K) \in \varpi \OF^\times$. In fact, we can and will fix $t_K$ as follows:
\begin{equation}\label{e:tk}
 t_K = \left[\begin{smallmatrix}b'/2&c\\-a & b'/2 \end{smallmatrix}\right],\quad \text{ where } b' = \begin{cases}0  &\text{ if } d/4 \notin \OF^\times, \\  2 &\text{ if } d/4 \in \OF^\times.\end{cases}
\end{equation}

\begin{lemma}\label{lemmaJ0phiLam} Suppose that $\Lambda$ is an unramified character. Then $J_0(\phi, \Lambda) = J_0(\phi, 1) + \Lambda^{-1}(\varpi_K)J_0(\phi, t_K)$.
\end{lemma}
\begin{proof}This follows immediately from \eqref{e:defj0} and the decomposition $F^\times \bs T = \OF^\times \bs T(\OF) \sqcup  \OF^\times \bs(t_K T(\OF))$.
\end{proof}
\begin{proposition}\label{prop:ramified-J0-before-simplification}
 Let $b'$ be as in \eqref{e:tk}. Then there exist $u_1,u_2\in\OF^\times$ such that
 $$
  J_0^{(k)}(\phi, 1)=  \int\limits_{x,y,z \in \p^{-k}}\Phi_0(\left[\begin{smallmatrix} 1 & & x & y \\ & 1 & y & z \\ & & 1 & \\ & & & 1 \end{smallmatrix}\right])\psi^{-1}(u_1\varpi x -b'y + u_2z)\,dx\,dy\,dz,
 $$
 $$
  J_0^{(k)}(\phi, t_K) =\int\limits_{x,y,z \in \p^{-k}}\Phi_0(\left[\begin{smallmatrix} 1 & & x & y \\ & 1 & y & z \\ & & 1 & \\ & & & 1 \end{smallmatrix}\right]\left[\begin{smallmatrix} 1 &  \\ & \varpi  \\ & & \varpi \\ & & & 1 \end{smallmatrix}\right])\psi^{-1}(u_1\varpi x - b'y +u_2z)\,dx\,dy\,dz.
 $$
\end{proposition}
\begin{proof}In order to treat both equalities simultaneously, put $t=1$ or $t_K$, and think of $t$ as an element of $G(F)$ via the embedding defined earlier. Put $\kappa = \left[\begin{smallmatrix} 1 \\ c^{-1} & 1\end{smallmatrix}\right]$ if $d/4 \in \OF^\times$ and $\kappa=1$ otherwise. Put $w = \left[\begin{smallmatrix} 0 &-1\\ 1 & 0\end{smallmatrix}\right]$, $A= \left[\begin{smallmatrix} \kappa \\ & {}^t \kappa^{-1} \end{smallmatrix}\right]$, $A' = \left[\begin{smallmatrix} {}^t \kappa^{-1}  \\ & \kappa  \end{smallmatrix}\right]$, $B= \left[\begin{smallmatrix} w  \\ & w \end{smallmatrix}\right]$. Note that $A, B, A'$ are elements of $\Sp(4,\OF)$. Put $X= \mat{x}{y}{y}{z} $, $X' = \kappa X {}^t \kappa$ and  $S' = {}^t \kappa^{-1} \mat{a}{}{}{c}\kappa^{-1}  = \mat{\varpi u_1}{-b'/2}{-b'/2}{u_2}$ where $u_1$, $u_2$ are in $\OF^\times$. Now we have
\begin{align*}
 J_0^{(k)}(\phi, t)&=  \int\limits_{x,y,z \in \p^{-k} }\Phi_0(\left[\begin{smallmatrix} I_2 & X \\ &I_2\end{smallmatrix}\right] t)\psi^{-1}(a x +c z)\,dx\,dy\,dz,\\
 & = \int\limits_{x,y,z \in \p^{-k} }\Phi_0(\left[\begin{smallmatrix} I_2 & X' \\ &I_2\end{smallmatrix}\right]   A t B A')\psi^{-1}(\trace(SX))\,dx\,dy\,dz, \\
 & = \int\limits_{x,y,z \in \p^{-k} }\Phi_0(\left[\begin{smallmatrix} I_2 & X \\ &I_2\end{smallmatrix}\right]   A t B A')\psi^{-1}(\trace(S' X))\,dx\,dy\,dz,\\
 & = \int\limits_{x,y,z \in \p^{-k} }\Phi_0(\left[\begin{smallmatrix} I_2 & X \\ &I_2\end{smallmatrix}\right]   A t B A')\psi^{-1}(u_1\varpi x -b'y + u_2z)\,dx\,dy\,dz.
\end{align*}
The result now follows from the observation that $A t B A'$ is a diagonal matrix of the desired form up to right multiplication by an element of $G(\OF)$.
\end{proof}

Henceforth, for convenience, denote $\Phi_0 = \Phi_\phi$. In particular, $\Phi_0$ is bi-$Z(F)G(\OF)$ invariant and satisfies $\Phi_0(1)=1$.
For non-negative integers $\ell$, $m$, let
$$
 h(\ell,m) = \left[\begin{smallmatrix} \varpi^{\ell + 2m}&&&\\&\varpi^{\ell+m}&&\\&&1&\\&&&\varpi^m\end{smallmatrix}\right].
$$
Due to the disjoint double coset decomposition, $$G = \bigsqcup_{\ell \ge 0, m\ge 0}Z(F)G(\OF)h(\ell,m)G(\OF)$$ and the fact that $\Phi_0$ is bi-$G(\OF)$ invariant, it follows that $\Phi_0$ is determined by its values at the matrices $h(\ell,m)$.
The following special case of \cite[Proposition 3.2]{knightlyli2016} will be useful for us.
\begin{lemma}\label{matrixcoset}
Let $g = \left[\begin{smallmatrix} 1 & & x & y \\ & 1 & y & z \\ & & 1 \\ & & & 1 \end{smallmatrix}\right]\left[\begin{smallmatrix} 1&  \\ & u  \\ & & u \\ & & & 1 \end{smallmatrix}\right]$ for some $x,y,z \in F$, $u \in \OF$. Define
 $$
  \ell = v(u) + 2\left(\min\left(0, v(ux),  v(y), v(z)\right) - \min\left(0, v(u) + v(xz-y^2), v(ux), v(uy), v(z)\right)\right),
 $$
 $$
  m = \min\left(0, v(u) + v(xz-y^2), v(ux), v(uy), v(z)\right) - 2\min\left(0, v(ux),v(y), v(z)\right).
 $$
 Then $g \in  Z(F)G(\OF)h(\ell,m)G(\OF).$ In particular, $\Phi_0(g)= \Phi_0(h(\ell,m))$.
\end{lemma}
We now come to a key result.
\begin{proposition}\label{prop:ramified-J0-final-values}  We have
\begin{itemize} \item $J_0(\phi, 1) = 1 - \Phi_0(h(0, 1)) - q^2\Phi_0(h(0, 2)) + q^2\Phi_0(h(2, 1))$.
\item $J_0(\phi, t_K) = q\Phi_0(h(1, 0)) - (q^2 + q)\Phi_0(h(1, 1)) + q^2\Phi_0(h(3, 0))$.\end{itemize}\end{proposition}

\begin{proof}
It will be convenient for us to define the sets, for $m \geq 0$,
$$
 \Pi_m = \begin{cases} \varpi^{-m}\OF^\times & \text{if } m >0,\\ \OF & \text{if }m=0.\end{cases}
$$
We assume throughout the proof that $k \ge 2$.

\textbf{Computing $J_0^{(k)}(\phi, 1)$.}  Let $m, n, r$ be non-negative integers.  First assume $m + r \neq 2n$.  If $(x, y, z) \in \Pi_m \times \Pi_n \times \Pi_r$ then, using the formula of Proposition \ref{prop:ramified-J0-before-simplification} and then applying Lemma \ref{matrixcoset} (which implies in particular that in this case, the value of $\Phi_0$ depends only on $m$, $n$, $r$) we are reduced to computing a character integral
$$
 \int\limits_{x \in \Pi_m} \int\limits_{y \in \Pi_n} \int\limits_{z \in\Pi_r} \psi^{-1}(u_1\varpi x -b'y + u_2z) \; dz \; dy \; dx.
$$
The above integral is zero if $m > 2$ or $r > 1$. Thus, these contributions vanish. If $m=2$ and $n \geq 2$, the same argument shows that the two integrals for $r=0$ and $r=1$ will sum to give zero. Similarly, we dispose of the case $m=1$, $n \geq 2$, and $m=0$, $n \geq 1$.

Next, assume $m + r = 2n$. Conjugating the argument of $\Phi_0$ by $\diag(1, 1, \lambda, \lambda)$ with $\lambda \in \OF^\times$ and making a change of variables the integral in question becomes
$$
 \int\limits_{x \in \Pi_m} \int\limits_{y \in \Pi_n} \int\limits_{z \in \Pi_r} \Phi_0(\left[\begin{smallmatrix} 1 & & x & y \\ & 1 & y & z \\ & & 1 \\ & & & 1 \end{smallmatrix}\right])\psi^{-1}(\lambda(u_1\varpi x - b' y + u_2z))\;dz\;dy\;dx.
$$
Choosing $\lambda = 1 + pw$ with $w \in \OF$ and integrating both sides of this equality over $w \in \OF$ we get zero whenever $n\ge 2$. From the above, we get
$$
 \begin{aligned} J_0^{(k)}(\phi, 1) =  \sum_{(m, n, r) \in S}\:\int\limits_{x \in \Pi_m} \int\limits_{y \in \Pi_n} \int\limits_{z \in \Pi_r} \Phi_0(\left[\begin{smallmatrix} 1 & & x & y \\ & 1 & y & z \\ & & 1 \\ & & & 1 \end{smallmatrix}\right])\psi^{-1}(u_1\varpi x - b' y + u_2z)\;dz\;dy\;dx,\end{aligned},
$$
where
$$
 S = \{ (2, 1, 1), (2, 0, 1), (2, 0, 0), (1, 1, 0), (1, 0, 1), (1, 0, 0), (0, 0, 1), (2,1,0), (1,1,1), (0,1,2), (0,0,0,)\}.
$$
The values of the integrals for tuples in $S$ are easily computed using Lemma \ref{matrixcoset}; summarized in the following table:\begin{footnote}{The ``Coefficient'' column computes the value of the integral $\int_{x \in \Pi_m} \int_{y \in \Pi_n} \int_{z \in \Pi_r} \psi^{-1}(u_1\varpi x - b' y + u_2z)$.}\end{footnote}

\begin{center}
\begin{tabular}{l|l|l|l|l}
$(m, n, r)$ & $h(\ell, m)$ & Coefficient \\ \hline
$(2, 1, 1)$ & $h(2, 1)$ & $q^2-q$ \\
$(2, 0, 1)$ &  $h(2, 1)$ & $q$ \\
$(2, 0, 0)$ & $h(0, 2)$ & $-q$ \\
$(1, 1, 0)$ & $h(2, 0)$ & $q^2-2q+1$ \\
$(1, 0, 1)$ &  $h(2, 0)$ & $-q+1$ \\
$(1, 0, 0)$ & $h(0, 1)$ & $q-1$ \\
\end{tabular}
\qquad
\begin{tabular}{l|l|l|l|l|l}
$(m, n, r)$ & $v(z-y^2x^{-1})$& $h(\ell, m)$ & Coefficient \\ \hline
$(2, 1, 0)$ &   & $h(0, 2)$ & $-q^2 + q$\\
$(1, 1, 1)$ & $-1$ &  $h(2, 0)$ & $-q^2 + 3q - 2$ \\
 & $\geq 0$  & $h(0, 1)$ & $-q+1$ \\
$(0, 1, 2)$ &  & $h(0, 2)$ & $0$ \\
$(0, 0, 0)$ &  & $h(0, 0)$ & $1$ \\
$(0, 0, 1)$ &  & $h(0,1)$ & -1 \\
\end{tabular}
\end{center}
which gives
$$
 J_0^{(k)}(\phi, 1) = \Phi_0(1_4) - \Phi_0(h(0, 1)) - q^2\Phi_0(h(0, 2)) + q^2\Phi_0(h(2, 1)).
$$

\textbf{Computing $J_0^{(k)}(\phi, t_K)$.}  The argument is essentially the same. Once again we get
$$
 \begin{aligned} J_0^{(k)}(\phi, t_K) =  \sum_{(m, n, r) \in S}\:\int\limits_{x \in \Pi_m} \int\limits_{y \in \Pi_n} \int\limits_{z \in \Pi_r} \Phi_0(\left[\begin{smallmatrix} 1 & & x & y \\ & 1 & y & z \\ & & 1 \\ & & & 1 \end{smallmatrix}\right]\left[\begin{smallmatrix} 1 &  \\ & \varpi  \\ & & \varpi \\ & & & 1 \end{smallmatrix}\right])\psi^{-1}(u_1\varpi x - b' y + u_2z)\;dz\;dy\;dx,\end{aligned}
$$
where
$$
 S = \{ (2, 1, 1), (2, 0, 1), (2, 0, 0), (1, 1, 0), (1, 0, 1), (1, 0, 0), (0, 0, 1), (2,1,0), (1,1,1), (0,1,2), (0,0,0,)\}.
$$
The values of the integrals for tuples in $S$ are summarized in the following table:

\begin{center}
\begin{tabular}{l|l|l|l|l}
$(m, n, r)$ & $h(l, m)$ & Coefficient \\ \hline
$(2, 1, 1)$ & $h(3, 0)$ & $q^2-q$ \\
$(2, 0, 1)$ &  $h(3, 0)$ & $q$ \\
$(2, 0, 0)$ & $h(1, 1)$ & $-q$ \\
$(1, 1, 0)$ & $h(1, 1)$ & $q^2-2q+1$ \\
$(1, 0, 1)$ &  $h(1, 1)$ & $-q+1$ \\
$(1, 0, 0)$ & $h(1, 0)$ & $q-1$ \\
\end{tabular}
\qquad
\begin{tabular}{l|l|l|l|l|l}
$(m, n, r)$ & $v(z-y^2x^{-1})$& $h(l, m)$ & Coefficient \\ \hline
$(2, 1, 0)$ &   & $h(1, 1)$ & $-q^2 + q$\\
$(1, 1, 1)$ & $-1$ &  $h(1, 1)$ & $-q^2 + 3q - 2$ \\
 & $\geq 0$  & $h(1, 1)$ & $-q+1$ \\
$(0, 1, 2)$ &  & $h(1, 2)$ & $0$ \\
$(0, 0, 0)$ &  & $h(1, 0)$ & $1$ \\
$(0, 0, 1)$ &  & $h(1,1)$ & -1 \\
\end{tabular}
\end{center}
which gives
$J_0^{(k)}(\phi, t_K) =  q^2h(3, 0) - (q^2 + q) h(1, 1) + qh(1, 0).$
\end{proof}

\emph{Henceforth, assume that $\pi$ is either a type I or a type IIb representation.}
We fix the following notation for Satake parameters. If $\pi = \chi_1\times\chi_2\rtimes\sigma$ is a type I representation, put $\alpha=\chi_1(\varpi)$, $\beta=\chi_2(\varpi)$, $\gamma=\sigma(\varpi)$.
 If $\pi = \chi \triv_{\GL(2)}\rtimes\sigma$ is a type IIb representation, put $\alpha=q^{-\frac12}\chi(\varpi),$ $\beta=q^{\frac12}\chi(\varpi),$ $\gamma=\sigma(\varpi).$
In both cases we have that $\alpha \beta \gamma^2 =1$ and the degree 5 $L$-function for $\pi$ is given by
\begin{equation}\label{e:standard}
 L(s, \pi, \mathrm{Std}) = \left((1-q^{-s})(1-\alpha q^{-s})(1-\alpha^{-1} q^{-s})(1-\beta q^{-s})(1-\beta^{-1} q^{-s})\right)^{-1}.
\end{equation}
In the IIb case, we will assume in addition that $\chi\sigma$ is trivial (note that $\chi^2\sigma^2$ is trivial by the central character condition), i.e., that $\pi$ is of the form $\chi1_{\GL(2)}\rtimes\chi^{-1}$. It is only under this condition that $\pi$ admits a Bessel model for trivial $\Lambda$.

\begin{lemma}
\label{lem:ramified-J0-normalising-factors}
 Let $\pi$ be a type I representation, or a type IIb representation of the form $\chi1_{\GL(2)}\rtimes\chi^{-1}$. In the former case, let $\Lambda$ be an unramified character and put  $l=\Lambda(\varpi_K) \in \{1,-1\}$; in the latter case  define $l=1$. Define the quantity $C$ to be equal to
 $$
  \frac{(1+q^{-1})^2(1+q^{-2})}{(1+l\gamma \alpha q^{-\frac12}) (1+l\gamma q^{-\frac12}) (1+l\gamma ^{-1} q^{-\frac12}) (1+l\gamma \beta q^{-\frac12})(1-\alpha q^{-1}) (1-\alpha^{-1} q^{-1}) (1-\beta q^{-1}) (1-\beta^{-1} q^{-1})  }.
 $$
 Then, if $\pi$ is a type I representation then $J(\phi, \Lambda)= C J_0(\phi, \Lambda)$ and if $\pi$ is a type IIb representation, then $J^*(\phi) = 2CJ_0^*(0, \phi)$.
\end{lemma}
\begin{proof}This follows from the definitions \eqref{e:defj}, \eqref{e:defjqiu} and well-known formulas for the $L$-functions involved in terms of the Satake parameters.
\end{proof}

Next we write down explicitly a special case of Macdonald's formula (see \cite{Casselman1980}); the proof follows by a routine computation and is therefore omitted.

\begin{proposition}\label{GSp4macdonaldeq2}
 For all non-negative integers $\ell$, $m$, we have
 \begin{equation}
  \Phi_0(h(\ell,m))=\frac{q^{-(4m+3\ell)/2}}{1+2q^{-1}+2q^{-2}+2q^{-3}+q^{-4}}\,\sum_{i=1}^8 A_iB_i,
 \end{equation}
 where the quantities $A_i$, $B_i$ are given as follows:
 $$
 \begin{array}{ccccc}
  \toprule
   i&A_i&B_i\\
  \toprule
   1&\frac{1-q^{-1}\alpha^{-1}\beta}{1-\alpha^{-1}\beta}\:\frac{1-q^{-1}\beta^{-1}}{1-\beta^{-1}}\:\frac{1-q^{-1}\alpha^{-1}\beta^{-1}}{1-\alpha^{-1}\beta^{-1}}\:\frac{1-q^{-1}\alpha^{-1}}{1-\alpha^{-1}}&\alpha^{2m+\ell}\beta^{m+\ell}\gamma^{2m+\ell}\\
  \midrule
   2&\frac{1-q^{-1}\alpha\beta^{-1}}{1-\alpha\beta^{-1}}\:\frac{1-q^{-1}\alpha^{-1}}{1-\alpha^{-1}}\:\frac{1-q^{-1}\alpha^{-1}\beta^{-1}}{1-\alpha^{-1}\beta^{-1}}\:\frac{1-q^{-1}\beta^{-1}}{1-\beta^{-1}}&\alpha^{m+\ell}\beta^{2m+\ell}\gamma^{2m+\ell}\\
  \midrule
   3&\frac{1-q^{-1}\alpha^{-1}\beta^{-1}}{1-\alpha^{-1}\beta^{-1}}\:\frac{1-q^{-1}\beta}{1-\beta}\:\frac{1-q^{-1}\alpha^{-1}\beta}{1-\alpha^{-1}\beta}\:\frac{1-q^{-1}\alpha^{-1}}{1-\alpha^{-1}}&\alpha^{2m+\ell}\beta^{m}\gamma^{2m+\ell}\\
  \midrule
  4&\frac{1-q^{-1}\alpha\beta}{1-\alpha\beta}\:\frac{1-q^{-1}\alpha^{-1}}{1-\alpha^{-1}}\:\frac{1-q^{-1}\alpha^{-1}\beta}{1-\alpha^{-1}\beta}\:\frac{1-q^{-1}\beta}{1-\beta}&\alpha^{m+\ell}\gamma^{2m+\ell}\\
  \midrule
   5&\frac{1-q^{-1}\alpha^{-1}\beta^{-1}}{1-\alpha^{-1}\beta^{-1}}\:\frac{1-q^{-1}\alpha}{1-\alpha}\:\frac{1-q^{-1}\alpha\beta^{-1}}{1-\alpha\beta^{-1}}\:\frac{1-q^{-1}\beta^{-1}}{1-\beta^{-1}}&\alpha^{m}\beta^{2m+\ell}\gamma^{2m+\ell}\\
  \midrule
  6&\frac{1-q^{-1}\alpha\beta}{1-\alpha\beta}\:\frac{1-q^{-1}\beta^{-1}}{1-\beta^{-1}}\:\frac{1-q^{-1}\alpha\beta^{-1}}{1-\alpha\beta^{-1}}\:\frac{1-q^{-1}\alpha}{1-\alpha}&\beta^{m+\ell}\gamma^{2m+\ell}\\
  \midrule
 7&\frac{1-q^{-1}\alpha^{-1}\beta}{1-\alpha^{-1}\beta}\:\frac{1-q^{-1}\alpha}{1-\alpha}\:\frac{1-q^{-1}\alpha\beta}{1-\alpha\beta}\:\frac{1-q^{-1}\beta}{1-\beta}&\alpha^{m}\gamma^{2m+\ell}\\
  \midrule
  8&\frac{1-q^{-1}\alpha\beta^{-1}}{1-\alpha\beta^{-1}}\:\frac{1-q^{-1}\beta}{1-\beta}\:\frac{1-q^{-1}\alpha\beta}{1-\alpha\beta}\:\frac{1-q^{-1}\alpha}{1-\alpha}&\beta^{m}\gamma^{2m+\ell}\\
  \bottomrule
 \end{array}
 $$
 \end{proposition}
\bigskip

\noindent\textbf{\emph{Proof of Theorem \ref{t:localunram}.}} This follows by combining Lemma \ref{J0*rel}, Lemma \ref{lemmaJ0phiLam}, Proposition \ref{prop:ramified-J0-final-values}, Lemma \ref{lem:ramified-J0-normalising-factors} and Proposition \ref{GSp4macdonaldeq2}.

\subsection{Basic structure theory}\label{basicstructuresec}

In this subsection we recall some of the structure theory of $G(F)$ concerning Weyl groups, parahoric subgroups, and the Iwahori-Hecke algebra. Many of the stated facts can be verified directly. For others, we refer to \cite{Borel1976} and the references therein.
\subsubsection*{Weyl groups}
Let $D$ be the diagonal torus of ${\rm Sp}(4)$. Let $\tilde N$ be the normalizer of
$D$ in ${\rm Sp}(4)$. Let
\begin{equation}\label{Weylgroupseq}
 \tilde W=\tilde N(F)/D(\OF)\qquad\text{and}\qquad W=\tilde N(F)/D(F).
\end{equation}
The Weyl group $W$ has $8$ elements and is generated by the images of  $$s_1 =  \left[\begin{smallmatrix}&1\\1\\&&&1\\&&1\end{smallmatrix}\right], \quad s_2 = \left[\begin{smallmatrix}&&1\\&1\\-1\\&&&1\end{smallmatrix}\right].$$
There is an exact sequence
\begin{equation}\label{Wexactseqeq}
 1\longrightarrow D(F)/D(\OF)\longrightarrow\tilde W\longrightarrow W\longrightarrow1,
\end{equation}
where $D(F)/D(\OF)\cong\Z^2$ via the map ${\rm diag}(a,b,a^{-1},b^{-1})\mapsto(v(a),v(b))$. The \emph{Atkin-Lehner element} of $G(F)$ is defined by
\begin{equation}\label{ALdefeq2}
 \eta=\left[\begin{smallmatrix}&&&-1\\&&1\\&\varpi\\-\varpi\end{smallmatrix}\right]=
 s_2s_1s_2\left[\begin{smallmatrix}\varpi\\&-\varpi\\&&1\\&&&-1\end{smallmatrix}\right]=
 \left[\begin{smallmatrix}-1\\&1\\&&-\varpi\\&&&\varpi\end{smallmatrix}\right]s_2s_1s_2.
\end{equation}
Set
\begin{equation}\label{s0defeq}
 s_0=\eta s_2\eta^{-1}=\left[\begin{smallmatrix}1&&&\\&&&\varpi^{-1}
 \\&&1\\&-\varpi\end{smallmatrix}\right].
\end{equation}
Then $\tilde W$ is generated by the images of $s_0,s_1,s_2$. (We will not distinguish in notation between the matrices $s_0,s_1,s_2$ and their images in the Weyl groups.) There is a length function on the Coxeter group $\tilde W$, which we denote by $\ell$.
\subsubsection*{Parahoric subgroups}
The Iwahori subgroup is defined as
\begin{equation}\label{Idefeq}
 I=G(\OF)\cap\left[\begin{smallmatrix}
    \OF&\p&\OF&\OF\\\OF&\OF&\OF&\OF\\\p&\p&\OF&\OF\\\p&\p&\p&\OF
   \end{smallmatrix}\right].
\end{equation}
If $J\subsetneq\{0,1,2\}$, then the corresponding standard parahoric subgroup $P_J$ is generated by $I$ and $s_j$, where $j$ runs through $J$. More precisely,
\begin{equation}\label{PJdecompeq}
 P_J=\bigsqcup_{w\in\langle s_j:\:j\in J\rangle}IwI.
\end{equation}
 The parahoric corresponding to $\{1\}$ is called the \emph{Siegel congruence subgroup} and denoted by $P_1$. The parahoric corresponding to $\{2\}$ is called the \emph{Klingen congruence subgroup} and denoted by $P_2$. Explicitly,
\begin{equation}\label{congruencesubgroupseq}
 P_1=G(\OF)\cap\left[\begin{smallmatrix}\OF&\OF&\OF&\OF\\\OF&\OF&\OF&\OF\\\p&\p&\OF&\OF\\\p&\p&\OF&\OF\end{smallmatrix}\right],\qquad
 P_2=G(\OF)\cap\left[\begin{smallmatrix}\OF&\p&\OF&\OF\\\OF&\OF&\OF&\OF\\\OF&\p&\OF&\OF\\\p&\p&\p&\OF\end{smallmatrix}\right].
\end{equation}
Note that $P_1$ is normalized by $\eta$, but $P_2$ is not. We let $P_0:=\eta P_2\eta^{-1}$; this is the parahoric corresponding to $\{0\}$. The parahoric corresponding to $\{1,2\}$ is $K:=G(\OF)$.

\subsubsection*{The Iwahori-Hecke algebra}
We recall the structure of the Iwahori-Hecke algebra $\mathcal{I}$, which is the convolution algebra of compactly supported left and right $I$-invariant functions on $G(F)$. Explicitly, for $T$ and $T'$ in $\mathcal{I}$, their product is given by
$$
 (T\cdot T')(x)=\int\limits_{G(F)}T(xy^{-1})T'(y)\,d^Iy.
$$
Here, $d^Iy$ is the Haar measure on $G(F)$ which gives $I$ volume $1$. The characteristic function of $I$ is the identity element of $\mathcal I$; we denote it by $e$.  The characteristic function of $\eta I$ is an element of $\mathcal I$, which we denote again by $\eta$. For $j=0,1,2$ let $e_i$ be the characteristic function of $Is_i I$. Then $\mathcal{I}$ is generated by $e_0,e_1,e_2$ and $\eta$.
For $w\in\tilde W$, let $q_w=\#IwI/I$. It is easy to verify that
\begin{equation}\label{qsieq}
 q_{s_i}=q\qquad\text{for }i=0,1,2.
\end{equation}
It is known that
\begin{equation}\label{qw1w2eq}
 q_{w_1w_2}=q_{w_1}q_{w_2}\qquad\text{if}\quad\ell(w_1w_2)=\ell(w_1)+\ell(w_2).
\end{equation}
For $w\in\tilde W$, let $T_w\in\mathcal{I}$ be the characteristic function of $IwI$. It is known that
\begin{equation}\label{Tw1w2eq}
 T_{w_1w_2}=T_{w_1}\cdot T_{w_2}\qquad\text{if}\quad\ell(w_1w_2)=\ell(w_1)+\ell(w_2).
\end{equation}

\subsection*{Action of the Iwahori-Hecke algebra on smooth representations}
The Iwahori-Hecke algebra $\mathcal{I}$ acts  on our representation $(\pi, V)$ by
$$
 Tv=\int\limits_{G(F)}T(g)\pi(g)v\,d^Ig,\qquad T\in\mathcal{I},\;v\in V.
$$
We denote by $V^I$ the subspace of $I$-invariant vectors. The action of $\mathcal{I}$ induces an endomorphism of $V^I$. 
If $V$ is irreducible, then so is $V^I$.

Set
\begin{equation}\label{didefeq}
 d_i=\frac1{q+1}(e+e_i),\qquad i\in\{0,1,2\}.
\end{equation}
Then $d_i^2=d_i$, and as operator on $V$, the element $d_i$ acts as a projection onto the space of fixed vectors $V^{P_i}$. We refer to the operator $d_1$ as \emph{Siegelization}, and to $d_2$ as \emph{Klingenization}.

\subsection*{Our goal}
Our goal for the rest of \S\ref{s:localint} is to compute the quantities $J(\phi, \Lambda)$ for suitable vectors $\phi \in \pi$ in certain cases where $\phi$ is not the spherical vector in a type I representation. Specifically, we will cover the cases where $\pi$ has a $P_1$-fixed vector, and we will take $\phi$ to be such a vector. If $\pi$ is spherical (i.e., has a $K$-fixed vector), we will assume it is generic; no such assumption will be made for the remaining representations.

Let us look more closely at the relevant representations $\pi$. Looking at the last two columns of  Table \eqref{Tablelocal}, we see that such a $\pi$ must be one of types I,  IIa, IIIa, Vb/c, VIa or VIb. (Note that the representations IVb, IVc are not unitary). For each of these representations $\pi$, we will choose $\phi$ to be any member of a suitable orthogonal basis for the space of $P_1$-fixed vectors. We will evaluate $J(\phi, \Lambda)$ exactly in each of these cases under certain additional assumptions which are listed later.

\subsection{Calculation of matrix coefficients}\label{matcoeffsec}
We consider the endomorphisms of the space $V^{P_1}$ induced by the elements
\begin{equation}\label{matcoeffcalceq1}
 d_1e_1e_0e_1\qquad\text{and}\qquad d_1e_0e_1e_0
\end{equation}
of the Iwahori-Hecke algebra $\mathcal{I}$. (Since the element $e_0e_1e_0$ commutes with $e_1$, we could have omitted the projection $d_1$ on the second operator, but we will carry it along for symmetry.) In this section we will calculate the matrix coefficients
\begin{equation}\label{matcoeffcalceq2}
 \lambda(\phi):=\frac{\langle d_1e_1e_0e_1\phi,\phi\rangle}{\langle\phi,\phi\rangle}\qquad\text{and}\qquad\mu(\phi):=\frac{\langle d_1e_0e_1e_0\phi,\phi\rangle}{\langle\phi,\phi\rangle}
\end{equation}
for certain $\pi$ and certain $\phi\in V^{P_1}$. The results will be needed as input for the calculation of the quantity $J_0(\phi, \Lambda)$ in the next subsection. Evidently, the matrix coefficients \eqref{matcoeffcalceq2} depend neither on the normalization of the inner product, nor on the normalization of the vector $\phi$.

All irreducible, admissible $(\pi,V)$ for which $V^{P_1}$ is not zero can be realized as subrepresentations of a full induced representation $\chi_1\times\chi_2\rtimes\sigma$ with unramified $\chi_1,\chi_2,\sigma$. Since we are working with a version of $\GSp(4)$ which is different from the one in \cite{NF}, it is necessary to clarify the notation. 
Let $$
 G'=\{g\in\GL(4):\:^tgJ'g=\lambda J',\:\lambda\in\GL(1)\},\qquad J'=\left[\begin{smallmatrix}&&&1\\&&1\\&-1\\-1\end{smallmatrix}\right].
$$
Then $G'$ is the \emph{symmetric} version of $\GSp(4)$; this is the one used in \cite{NF}. There is an isomorphism between $G(F)$ and $G'(F)$ given by switching the first two rows and the first two columns. For example, the Iwahori subgroup $I$ defined in \eqref{Idefeq} corresponds to the subgroup $I'$ of $G'(\OF)$ consisting of matrices which are upper triangular mod $\p$. If $(\pi,V)$ is a representation of $G(F)$, let $(\pi',V)$ be the representation of $G'(F)$ obtained by composing $\pi$ with the isomorphism $G(F)\cong G'(F)$. This establishes a one-one correspondence between representations of $G(F)$ and representations of $G'(F)$. When we write $\chi_1\times\chi_2\rtimes\sigma$, we mean the representation of $G(F)$ that corresponds to the representation of $G'(F)$ which in \cite{NF} was denoted by the same symbol.

In the following we will rely, sometimes without mentioning it, on Table A.15 of \cite{NF}, which lists the dimensions of the spaces of fixed vectors under all parahoric subgroups for all Iwahori-spherical representations of $G(F)$.
\subsubsection*{Type I}
Let $\pi=\chi_1\times\chi_2\rtimes\sigma$ with unramified $\chi_1,\chi_2,\sigma$. Let $V$ be the standard model of $\pi$, and let $V^I$ be the subspace of $I$-invariant vectors. Then $V^I$ has the basis $f_w$, $w\in W$, where $f_w$ is the unique $I$-invariant element of $V$ with $f_w(w)=1$ and $f_w(w')=0$ for $w'\in W$, $w'\neq w$. It is convenient to order the basis as follows:
\begin{equation}\label{VIbasiseq2}
 f_e,\quad f_1,\quad f_2,\quad f_{21},\quad f_{121},\quad
 f_{12},\quad f_{1212},\quad f_{212},
\end{equation}
where we have abbreviated $f_1=f_{s_1}$ and so on. A basis for the four-dimensional space $V^{P_1}$ is given by
\begin{equation}\label{VP1basiseq}
 f_e+f_1,\qquad f_2+f_{21},\qquad f_{121}+f_{12},\qquad f_{1212}+f_{212}.
\end{equation}
Having fixed the basis \eqref{VIbasiseq2}, the operators $e_0,e_1,e_2$ and $\eta$ on $V^I$ become
$8\times8$--matrices. They depend on the Satake parameters
\begin{equation}\label{satakeeq}
 \alpha=\chi_1(\varpi),\qquad\beta=\chi_2(\varpi),\qquad\gamma=\sigma(\varpi)
\end{equation}
and are given\footnote{The matrix for $e_1$ in Lemma 2.1.1 of \cite{sch} contains a typo: The two lower right $2\times2$-blocks need to be conjugated by $\left[\begin{smallmatrix}0&1\\1&0\end{smallmatrix}\right]$.} in Lemma 2.1.1 of \cite{sch}.

Now assume that $\chi_1,\chi_2,\sigma$ are unitary characters, or equivalently, that $\alpha,\beta,\gamma$ have absolute value $1$. In this case $\chi_1\times\chi_2\rtimes\sigma$ is an irreducible representation of type I. It is unitary and tempered, with hermitian inner product on $V$ given by
\begin{equation}\label{matcoeffcalceq3}
 \langle f,f'\rangle=\int\limits_K f(g)\,\overline{f'(g)}\,dg.
\end{equation}
The vectors \eqref{VIbasiseq2} are orthogonal, since they are supported on disjoint cosets. Let $\phi_1,\ldots,\phi_4$ be the vectors in \eqref{VP1basiseq}, in this order. Then $\phi_1,\ldots,\phi_4$ is an orthogonal basis of $V^{P_1}$. The vector $\phi = \phi_1 + \phi_2 + \phi_3 + \phi_4$ is the $K$-fixed vector.
If we write
$d_1e_1e_0e_1\phi_i=\sum_{j=1}^4c_{ij}\phi_j,$ $d_1e_0e_1e_0\phi_i=\sum_{j=1}^4c_{ij}'\phi_j,$
then the quantities defined in \eqref{matcoeffcalceq2} are given by $\lambda(\phi_i)=c_{ii}$ and $\mu(\phi_i)=c'_{ii}$. Working out the linear algebra, we find that
\begin{equation}\label{matcoeffcalceq6}
 \lambda(\phi_1)=(q-1)q^2,\qquad \lambda(\phi_2)=\frac{q-1}{q+1}q^2,\qquad \lambda(\phi_3)=\frac{q-1}{q+1}q^2,\qquad \lambda(\phi_4)=0,
\end{equation}

\begin{equation}\label{matcoeffcalceq7}
 \mu(\phi_1)=(q-1)q^2,\qquad \mu(\phi_2)=(q-1)q,\qquad \mu(\phi_3)=0,\qquad \mu(\phi_4)=0.
\end{equation}
\subsubsection*{Type IIIa}
Let $\pi=\chi\rtimes\sigma\St_{\GSp(2)}$ with $\chi\notin\{1,\nu^{\pm2}\}$. Then $\pi$ is a representation of type IIIa. We assume $\chi\sigma^2=1$, so that $\pi$ has trivial central character. We realize $\pi$ as a subrepresentation of $\chi\times\nu\rtimes\nu^{-1/2}\sigma$, and set
$$
 \alpha=\chi(\varpi),\qquad\beta=q^{-1},\qquad\gamma=q^{1/2}\delta,
$$
where $\delta=\sigma(\varpi)$. Let $V$ be the standard space of the full induced representation $\chi\times\nu\rtimes\nu^{-1/2}\sigma$, and let $U$ be the subspace of $V$ realizing $\pi$. To determine the two-dimensional space $U^{P_1}$, we observe that $\dim U^{P_2}=1$, so that the Klingenization map $d_2:U^{P_1}\to U^{P_2}$ has a non-trivial kernel. The condition $d_2\phi=0$ characterizes a unique vector in $V^{P_1}$, which must thus lie in $U^{P_1}$. A second, linearly independent vector can then be obtained by applying $\eta$. Thus we find that $U^{P_1}$ is spanned by
\begin{equation}\label{IIIaeq32}
 \phi_1=\,^t[q,q,-1,-1,0,0,0,0]\qquad\text{and}\qquad\phi_2=\,^t[0,0,0,0,-q,-q,1,1];
\end{equation}
these are column vectors using the basis \eqref{VIbasiseq2}.
These vectors are eigenvectors for the endomorphism $T_{1,0}=e_2e_1e_2\eta$:
\begin{equation}\label{IIIaeq33}
 T_{1,0}\phi_1=\alpha\delta q \phi_1,\qquad T_{1,0}\phi_2=\delta q\phi_2.
\end{equation}
Now assume that $\chi$ and $\sigma$ are unitary. In this case $\pi$ is unitary and tempered. Since the invariant inner product $\langle\cdot,\cdot\rangle$ is not a priori given by formula \eqref{matcoeffcalceq3}, the following lemma is not obvious:
\begin{lemma}\label{IIIaorthogonallemma}
 The vectors $\phi_1$ and $\phi_2$ are orthogonal.
\end{lemma}
\begin{proof}
First note that
\begin{equation}\label{IIIaorthogonallemmaeq1}
 \langle\pi(\eta)v,w\rangle=\langle v,\pi(\eta)w\rangle\qquad\text{for all }v,w\in U;
\end{equation}
this is immediate, since $\eta^2$ acts as the identity on $V$. By Proposition 2.1.2 of \cite{sch}, we have
\begin{equation}\label{IIIaorthogonallemmaeq2}
 \langle e_iv,w\rangle=\langle v,e_iw\rangle\qquad\text{for all }v,w\in U^I\text{ and }i\in\{1,2\}.
\end{equation}
We will now calculate $\langle T_{1,0}\phi_1,\phi_2\rangle$ in two different ways. On the one hand, by \eqref{IIIaeq33},
\begin{equation}\label{IIIaorthogonallemmaeq3}
 \langle T_{1,0}\phi_1,\phi_2\rangle=\alpha\delta q\langle\phi_1,\phi_2\rangle.
\end{equation}
On the other hand, by \eqref{IIIaorthogonallemmaeq1} and \eqref{IIIaorthogonallemmaeq2}, $\langle T_{1,0}\phi_1,\phi_2\rangle=\langle\phi_1,\eta e_2e_1e_2\phi_2\rangle$. A calculation using the explicit form \eqref{IIIaeq32} shows that $\eta e_2e_1e_2\phi_2=\delta^{-1}q\phi_2$. Hence
\begin{equation}\label{IIIaorthogonallemmaeq4}
 \langle T_{1,0}\phi_1,\phi_2\rangle=\delta q\langle\phi_1,\phi_2\rangle.
\end{equation}
Our assertion follows from \eqref{IIIaorthogonallemmaeq3} and \eqref{IIIaorthogonallemmaeq4}, since $\alpha\neq1$ for representations of type IIIa.
\end{proof}

In view of Lemma \ref{IIIaorthogonallemma}, we can now calculate the quantities \eqref{matcoeffcalceq2} for the vectors $\phi_1$ and $\phi_2$ similar to the type I case. The result is
\begin{equation}\label{IIIaeq40}
 \lambda(\phi_i)=-\frac{q^2}{q+1},\qquad\mu(\phi_i)=0,
\end{equation}
for $i=1,2$.
\subsubsection*{Type VIb}
Consider the representation $\pi=\tau(T,\nu^{-1/2}\sigma)$ of type VIb. We will assume $\sigma^2=1$, so that $\pi$ has trivial central character. There is a unique $P_1$-fixed element (up to scalars) $\phi_1$ in the space of $\pi$. A similar calculation to the above shows that
\begin{equation}\label{VIbeq2}
 \lambda(\phi_1)=-q^2\qquad\text{and}\qquad\mu(\phi_1)=q.
\end{equation}
\subsubsection*{Types IIa, Vb/c and VIa}
Let $(\pi,V)$ be an Iwahori-spherical representation of type IIa, Vb, Vc or VIa. Then $V^{P_1}$ is one-dimensional by Table \eqref{Tablelocal}. Assume in addition that $\pi$ is unitary and has trivial central character. If $\phi_1$ is a vector spanning $V^{P_1}$, then calculations similar to the above show that
\begin{equation}\label{IIaVbVIaeq1}
 \lambda(\phi_1)=\frac{q-1}{q+1}q^2\qquad\text{and}\qquad\mu(\phi_1)=-q.
\end{equation}
We mention these representations just for completeness; they are not relevant for our global applications, since, by Theorem 6.2.2 of \cite{RobertsSchmidt2014}, they do not admit a special, non-split Bessel model.

\subsection{The local integral for certain ramified representations}
Our goal in this section is to prove the following analogue of Theorem \ref{t:localunram}.

\begin{theorem}\label{t:localintmain}
 Suppose that $\pi$ is an irreducible, admissible, unitary, Iwahori-spherical representation of $\GSp(4,F)$ with trivial central character. We assume that $\pi$ is one of the following types: I, IIa, IIIa, Vb/c, VIa or VIb. For type I, let the vectors $\phi_1,\ldots,\phi_4$ be the vectors in \eqref{VP1basiseq}, in this order. For type IIIa, let $\phi_1$ and $\phi_2$ be as given in \eqref{IIIaeq32}; in all other cases let $\phi_1$ be the essentially unique $P_1$-invariant vector. Suppose that the residual characteristic of $F$ is odd, $K=F(\sqrt{d})$ is the unramified quadratic extension of $F$, and $\Lambda$ is the trivial character on $K^\times$.\footnote{This is equivalent to saying that $\Lambda$ is an unramified character on $K^\times$, since $\Lambda|_{F^\times}=1$ and $K/F$ is an unramified field extension.} Then the quantities $J_0(\phi, \Lambda)$, $J(\phi, \Lambda)$, $J^*(\phi)$ are as given by the last two columns of the table below.

$$\renewcommand{\arraystretch}{.8}\renewcommand{\arraycolsep}{2ex}
 \begin{array}{cccc}
  \toprule
  \text{type}&\phi&J_0(\phi, \Lambda)&J(\phi, \Lambda)= J^*(\phi)\\
  \toprule
   \mathrm{I}&\phi_1&q^{-1}&q^{-1}L(1, \pi, \mathrm{Std})(1-q^{-4})\\
  \cmidrule{2-4}
   &\phi_2&1&L(1, \pi, \mathrm{Std})(1-q^{-4})\\
  \cmidrule{2-4}
   &\phi_3&q^{-1}&q^{-1}L(1, \pi, \mathrm{Std})(1-q^{-4})\\
  \cmidrule{2-4}
   &\phi_4&1&L(1, \pi, \mathrm{Std})(1-q^{-4})\\
  \midrule
   \mathrm{IIIa}&\phi_1&1+q^{-1}&(1+q^{-2})(1+q^{-1})\\
  \cmidrule{2-4}
   &\phi_2&1+q^{-1}&(1+q^{-2})(1+q^{-1})\\
    \midrule
    \mathrm{VIb}&\phi_1&2(1+q^{-1})&2(1+q^{-2})(1+q^{-1})\\
    \midrule
   \mathrm{IIa, Vb/c, VIa}&\phi_1&0&0\\
  \bottomrule
 \end{array}
$$
\end{theorem}

\medskip
\begin{remark} Above, $L(1, \pi, \mathrm{Std})$ denotes the special value of the degree 5 $L$-function, defined in \eqref{e:standard}.
\end{remark}
Throughout the rest of this subsection we assume that:
\begin{enumerate}
  \item The residual characteristic of $F$ is odd.
  \item $K=F(\sqrt{d})$ is the unramified quadratic extension of $F$. (Equivalently, $d:=b^2-4ac$ is in $\OF^\times$ but not in $\OF^{\times2}$. Since $1+\p\subset\OF^{\times2}$, this implies $a\in\OF^\times$.)
  \item $\Lambda=1$.
  \item $\phi$ is a $P_1$-invariant vector.
\end{enumerate}
Note that the assumptions imply that \begin{equation}\label{e:j0eq}J_0(\phi, \Lambda) = J_0(\phi, 1).\end{equation} As in the proof of Theorem \ref{t:localunram}, we may assume that $b=0$. Also, as before, we use $\Phi_0$ to denote $\Phi_\phi$. Note that $\Phi_0$ is a $ZP_1$-bi-invariant function. The following result is the analogue of Proposition \ref{prop:ramified-J0-final-values} in this case.
\begin{proposition}\label{p:j0mainlocal} We have \begin{equation}\label{J0P1eq2}
  J_0(\phi, \Lambda) = 1-(q+1)\Phi_0(g_1)+q\Phi_0(g_2),
 \end{equation}
 where
 \begin{equation}\label{g1g2defeq}
  g_1=\left[\begin{smallmatrix}\varpi^{-1}\\&1\\&&\varpi\\&&&1\end{smallmatrix}\right]s_2,\qquad g_2=\left[\begin{smallmatrix}\varpi^{-1}\\&\varpi^{-1}\\&&\varpi\\&&&\varpi\end{smallmatrix}\right]s_2s_1s_2.
 \end{equation}
\end{proposition}
\begin{proof}
By \eqref{e:j0eq}, $J_0(\phi, \Lambda) = J_0(\phi, 1)= \lim_{k \rightarrow \infty} J_0^{(k)}(\phi, 1)$. So we need to compute
$$
 J_0^{(k)}(\phi, 1) = \int\limits_{x,y,z \in \p^{-k}}\Phi_0(\left[\begin{smallmatrix} 1 & & x & y \\ & 1 & y & z \\ & & 1 & \\ & & & 1 \end{smallmatrix}\right])\psi^{-1}(ax + cz)\,dx dy dz.
$$
We assume $k \ge 1$. A similar argument to that in Proposition \ref{prop:ramified-J0-final-values} shows that $$
 J_0^{(k)}(\phi, 1)=\int\limits_{\p^{-1}}\,\int\limits_{\p^{-k}}\,\int\limits_{\p^{-1}}\Phi_0(\left[\begin{smallmatrix}1&&x&y\\&1&y&z\\&&1\\&&&1\end{smallmatrix}\right])\psi^{-1}(ax+cz)\,dz\,dy\,dx.
$$
Using the left invariance properties of $\Phi_0$, and a simple calculation using the identity
\begin{equation}\label{usefulidentityeq}
 \mat{1}{y}{}{1}=\mat{1}{}{y^{-1}}{1}\mat{y}{}{}{y^{-1}}\mat{}{1}{-1}{}\mat{1}{}{y^{-1}}{1},
\end{equation} one can check that
\begin{equation}\label{J0calclemma2eq1}
\int\limits_{\p^{-1}}\,\int\limits_{\varpi^{-\ell}\OF^\times}\,\int\limits_{\p^{-1}}\Phi_0(\left[\begin{smallmatrix}1&&x&y\\&1&y&z\\&&1\\&&&1\end{smallmatrix}\right])\psi^{-1}(ax+cz)\,dz\,dy\,dx=0\qquad\text{for }\ell\geq2.
\end{equation}
We conclude that
$$
J_0^{(k)}(\phi, 1)=\int\limits_{\p^{-1}}\,\int\limits_{\p^{-1}}\,\int\limits_{\p^{-1}}\Phi_0(\left[\begin{smallmatrix}1&&x&y\\&1&y&z\\&&1\\&&&1\end{smallmatrix}\right])\psi^{-1}(ax+cz)\,dz\,dy\,dx.
$$
Write $J_0^{(k)}(\phi, 1)=J_1+J_2$, where $J_1$ is the part where $x\in\OF$, and $J_2$ is the part where $x\in\varpi^{-1}\OF^\times$. Evidently,
$$
 J_1=\int\limits_{\p^{-1}}\,\int\limits_{\p^{-1}}\Phi_0(\left[\begin{smallmatrix}1&&&y\\&1&y&z\\&&1\\&&&1\end{smallmatrix}\right])\psi^{-1}(cz)\,dz\,dy.
$$
Let $J_{11}$ be the part where $y\in\OF$, and let $J_{12}$ be the part where $y\notin\OF$. Using \eqref{usefulidentityeq}, we have
$$
 J_{11} = \Phi_0(1)-\Phi_0(\left[\begin{smallmatrix}1\\&\varpi^{-1}\\&&1\\&&&\varpi\end{smallmatrix}\right]s_1s_2s_1).
$$
By a conjugation of the argument of $\Phi_0$ and the use of bi-$I$-invariance we see that, for any $w\in\OF$,
\begin{align*}
 J_{12}&=\int\limits_{\varpi^{-1}\OF^\times}\,\int\limits_{\p^{-1}}\Phi_0(\left[\begin{smallmatrix}1&&&y\\&1&y&z\\&&1\\&&&1\end{smallmatrix}\right])\psi^{-1}(cz)\,dz\,dy\\
 &=\int\limits_{\varpi^{-1}\OF^\times}\,\int\limits_{\p^{-1}}\Phi_0(\left[\begin{smallmatrix}1&&&y\\&1&y&z+2wy\\&&1\\&&&1\end{smallmatrix}\right])\psi^{-1}(cz)\,dz\,dy.
\end{align*}
The element $w$ may depend on $y$ and $z$. Choosing $w=-y^{-1}z/2$, we see $J_{12}=0$. Thus
$$
 J_1=J_{11}=\Phi_0(1)-\Phi_0(\left[\begin{smallmatrix}1\\&\varpi^{-1}\\&&1\\&&&\varpi\end{smallmatrix}\right]s_1s_2s_1).
$$
A similar (but slightly more involved) computation gives
$$
 J_2=J_{21}+J_{22}=-q\Phi_0(\left[\begin{smallmatrix}\varpi^{-1}\\&1\\&&\varpi\\&&&1\end{smallmatrix}\right]s_2)+q\Phi_0(\left[\begin{smallmatrix}\varpi^{-1}\\&\varpi^{-1}\\&&\varpi\\&&&\varpi\end{smallmatrix}\right]s_2s_1s_2s_1).
$$
This concludes the proof, using the $P_1$-invariance of $\Phi_0$.
\end{proof}
\begin{remark}
The proof of the above proposition used only  the bi-$ZP_1$-invariance of the function $\Phi_0$, just as the proof of Proposition \ref{prop:ramified-J0-final-values} relied only  on the bi-$ZG(\OF)$-invariance.
\end{remark}

\begin{remark}
Suppose that $\phi$ is no longer assumed to be $P_1$-invariant, but merely $I$-invariant. Then with some additional work one can show
\[\begin{aligned} J_0(\phi, \Lambda) = \frac{1}{1+q} &\left[1 + q\Phi_\phi(s_1) - q^2 \Phi_\phi(\left[\begin{smallmatrix} \varpi^{-1} & & & \\ & 1 & & \\ & & \varpi & \\ & & & 1 \end{smallmatrix}\right]s_2) - q\Phi_\phi(\left[\begin{smallmatrix} 1 & & & \\ & \varpi^{-1} & & \\ & & 1 & \\ & & & \varpi \end{smallmatrix}\right] s_1 s_2) \right.\\
&\left.\qquad - q\Phi_\phi(\left[\begin{smallmatrix} \varpi^{-1} & & & \\ & 1 & & \\ & & \varpi & \\ & & & 1 \end{smallmatrix}\right]s_2 s_1) - \Phi_\phi(\left[\begin{smallmatrix} 1 & & & \\ & \varpi^{-1} & & \\ & & 1 & \\ & & & \varpi \end{smallmatrix}\right]s_1 s_2 s_1) \right.\\
&\left.\qquad + q\Phi_\phi(\left[\begin{smallmatrix} \varpi^{-1} & & & \\ & \varpi^{-1} & & \\ & & \varpi & \\ & & & \varpi \end{smallmatrix}\right]s_2 s_1 s_2) + q^2 \Phi_\phi(\left[\begin{smallmatrix} \varpi^{-1} & & & \\ & \varpi^{-1} & & \\ & & \varpi & \\ & & & \varpi\end{smallmatrix}\right] s_1 s_2 s_1 s_2)\right].\end{aligned}\]
\end{remark}

The values $\Phi_0(g_i)$ appearing in Proposition \ref{p:j0mainlocal} depend on the representation $\pi$. We can convert an operator $\pi(g)$ on $V^I$ into an element of the Iwahori-Hecke algebra $\mathcal{I}$. Recall that $e$, the characteristic function of $I$, is the identity element of $\mathcal{I}$. For $g\in G(F)$, let $T_g$ be the characteristic function of $IgI$. We have the following easy lemma, whose proof we omit.
\begin{lemma}\label{IgIlemma}
 Let $(\pi,V)$ be a smooth representation of $G(F)$. Let $g$ be any element of $G(F)$. Then
 $
  e\circ\pi(g)=\frac1{\#IgI/I}\,T_g
 $
 as operators on $V^I$.
\end{lemma}

\noindent\textbf{\emph{Proof of Theorem \ref{t:localintmain}.}} Applying Lemma \ref{IgIlemma} to the elements $g_1$ and $g_2$ in \eqref{g1g2defeq}, we see that
\begin{equation}\label{J0phicalceq2}
 e\circ\pi(g_i)=\frac1{\#Ig_iI/I}\,{\rm char}(Ig_iI)
\end{equation}
as operators on $V^{P_1}$. Note that, as elements of $\tilde W$,
\begin{equation}\label{J0phicalceq3}
 g_1=s_1s_0s_1,\qquad g_2=s_0s_1s_0.
\end{equation}
From \eqref{qsieq}, \eqref{qw1w2eq}, \eqref{Tw1w2eq} and \eqref{J0phicalceq2}, we therefore get
\begin{equation}\label{J0phicalceq4}
 e\circ\pi(g_1)=q^{-3}\,e_1e_0e_1,\qquad e\circ\pi(g_2)=q^{-3}\,e_0e_1e_0.
\end{equation}
Substituting into \eqref{J0P1eq2}, we obtain the formula
\begin{equation}\label{J0formulaeq}
 J_0(\phi, \Lambda)=1-(q+1)q^{-3}\lambda(\phi)+q^{-2}\mu(\phi),
\end{equation}
with $\lambda(\phi)$ and $\mu(\phi)$ as defined in \eqref{matcoeffcalceq2}. (We can insert $d_1$ into the inner product because $\phi$ is assumed to be $P_1$-invariant.)

The quantities $\lambda(\phi)$ and $\mu(\phi)$ have been calculated in Sect.~\ref{matcoeffsec} for various vectors in a number of representations. Using these, we can now compute $J_0(\phi, \Lambda)$ in each case using \eqref{J0formulaeq}. Furthermore, writing $$M(\pi) = \frac{L(1, \pi, \Ad)L(1, \chi_{K/F}) }{\zeta_{F}(2)\zeta_{F}(4)L(1/2, \pi \times \AI(\Lambda^{-1}))},$$ we have $J(\phi, \Lambda) = M(\pi)J_0(\phi, \Lambda).$ We can write down $M(\pi)$ in each case using the relevant data from \cite{sch} and \cite{asgarischmidt08}. As a result, we can compute $J(\phi, \Lambda)$  in each case as well. The result is exactly the values in Theorem \ref{t:localintmain}. The proof of the fact that $J_0(\phi, \Lambda)$ equals $J^*(\phi)$ follows by an identical argument to Lemma \ref{J0*rel} (recall that we are assuming that $\Lambda$ is the trivial character).

\begin{remark} Note that for representations of type IIa, Vb/c and VIa we obtain $J_0(\phi_1)=0$ for the essentially unique $P_1$-invariant vector $\phi_1$. This is consistent with the fact that these representations do not admit a non-split Bessel functional with trivial character $\Lambda$ on the torus $T(F)$; see Theorem 6.2.2 of \cite{RobertsSchmidt2014}. Hence, such representations will not be relevant for our global applications to Siegel modular forms. \end{remark}

\begin{remark}\label{rem:jpdim}Note that whenever $J(\phi, \Lambda) \neq 0$, we have $\sum_i J(\phi_i) - 2  \ll q^{-1}$.
\end{remark}
\section{Global results}\label{s:global}
In this final section, we return to a global setup. We begin with some basic results on the correspondence between Siegel modular forms and automorphic representations on $G(\A)$. Then we move on to the relation between Bessel periods and Fourier coefficients and the classical reformulation of Conjecture~\ref{c:liu}. Finally, we work out several consequences of our explicit refinement of B\"ocherer's conjecture. We remind the reader that all our  measures on adelic groups are equal to the respective Tamagawa measures.
\subsection{Siegel modular forms and representations}\label{s:classicalrep}
We use the definitions and some basic properties of Siegel modular forms of degree 2 without proof; we refer the reader to \cite{sch} or \cite{sahapet} for details. Let $k$ and $N$ be positive integers with $N$ square-free. Let $\Gamma_0(N) \subseteq \Sp(4,\Z)$ denote the Siegel congruence subgroup of squarefree level $N$, defined in \eqref{defu1n}.  If $p|N$ is prime, set $P_{1,p}(N) = P_1$, the local analog of $\Gamma_0(N)$ defined in \eqref{P1DEF}; if $p \nmid N$ is prime, set $P_{1, p}(N) = \GSp(4,\Z_p)$. We let $S_k(\Gamma_0(N))$ denote the space of Siegel cusp forms of weight $k$ with respect to the group $\Gamma_0(N)$. For any $f \in S_k(\Gamma_0(N))$, define the Petersson inner product
\begin{equation}\label{eqn:petersson-def}
 \langle f, f\rangle = \frac{1}{[\Sp(4,\Z):\Gamma_0(N)]}\;\int\limits_{\Gamma_0(N) \bs \H_2} |f(Z)|^2 (\det Y)^{k - 3}\,dX\,dY.
 \end{equation}
The space $S_k(\Gamma_0(N))$ has a natural orthogonal (with respect to the Petersson inner product) decomposition $$S_k(\Gamma_0(N)) =S_k(\Gamma_0(N))^{\rm old}  \oplus S_k(\Gamma_0(N))^{\rm new}$$ into the newspace and the oldspace in the sense of \cite{sch}.

For any $\phi \in L^2(Z(\A)G(\Q) \bs G(\A))$, let $\langle \phi, \phi\rangle = \int_{Z(\A)G(\Q) \bs G(\A)} |\phi(g)|^2\,dg$.

\begin{proposition}\label{prop:equival}
 Suppose that $N$ is a squarefree integer and $k$ a positive integer. Let $\pi$ be an irreducible cuspidal automorphic representation of $G(\A)$ with trivial central character. Suppose that\footnote{See the paper \cite{PSS14} for the definition of $L(k, \ell)$; this was also called $\mathcal{E}(k, \ell)$ previously by us.} $\pi_\infty \simeq L(k,k)$ and $\pi_p$ has a non-zero $P_{1,p}(N)$-fixed vector for all primes $p|N$. Let $\phi \in V_\pi \subset L^2(Z(\A)G(\Q) \bs G(\A))$ be a non-zero function such that\footnote{Our conditions on $\pi$ imply that such a function always exists.} $\phi(g k_Nk_\infty) = \det(J(k_\infty, iI_2))^{-k} \phi(g)$ for all $k_\infty \in K_\infty$, $k_N \in \prod_{p<\infty}P_{1,p}(N)$. Define the function $f$ on $\H_2$ by $f(g_\infty i) = \det(J(g_\infty, iI_2))^{k} \phi(g_\infty)$. Then $f$ is a non-zero element of $S_k(\Gamma_0(N))$ and is an eigenfunction for the local Hecke algebras at all $p \nmid N$. The function $f$ belongs to  $S_k(\Gamma_0(N))^{\rm old}$ if and only if $\pi_p$ is spherical for some $p|N$; else it belongs to $S_k(\Gamma_0(N))^{\rm new}$. Furthermore,
 $$
  \frac{\langle f, f\rangle}{\vl(\Sp(4,\Z)\bs \H_2)} = \frac{\langle \phi , \phi \rangle}{\vl(Z(\A)G(\Q) \bs G(\A))}.
 $$
\end{proposition}
\begin{proof}This is routine; see Sect.~3.2 of \cite{sch} and Sect.~4.1 of \cite{asgsch}.
\end{proof}

Given any $f \in S_k(\Gamma_0(N)),$ we can define its adelization as in \cite{sch} or \cite{sahapet}. A non-zero $f \in S_k(\Gamma_0(N))$ arises via Proposition \ref{prop:equival} if and only if its adelization generates an \emph{irreducible} automorphic representation, in which case the adelization precisely equals the function $\phi$ in the proposition. It is easy to see that the set of such $f$ spans the space $S_k(\Gamma_0(N))$, a fact that follows immediately from the decomposition of the cuspidal automorphic spectrum into a direct sum of irreducible representations.

Let $\pi$, $\phi$, $f$ be as in Proposition \ref{prop:equival}. Then using the results  of \cite{sch} (or see table \eqref{Tablelocal} earlier) we see that the local representations $\pi_p$ at finite primes have the following properties.

\begin{itemize}
 \item If $p \nmid N$, then $\pi_p$ is spherical, i.e., one of types I, IIb, IIIb, IVd, Vd, VId. Of these, types IIb, IIIb, IVd, Vd, VId are non-tempered (as well as non-generic). Type I is generic; moreover it is tempered provided the inducing characters are unitary.
 \item If $p|N$, then $\pi_p$ is either spherical (in which case it is one of those from the above list) or non-spherical, in which case it is one of types  IIa,  IIIa, Vb/c,  VIa or VIb. Of these, Vb/c are non-generic and non-tempered, VIb is non-generic and tempered, IIIa/VIa are generic and tempered, and IIa is generic and can be either tempered or non-tempered. Recall that conjecturally, $\pi_p$ is always tempered.

\end{itemize}
A representation $\pi$ of $G(\A)$ is said to be CAP if it is nearly equivalent to a constituent of a global induced representation of a proper parabolic subgroup of $G(\A)$. Otherwise we say that it is non-CAP. A CAP representation can be either P-CAP (associated to the Siegel parabolic) or Q-CAP (associated to the Klingen parabolic) or B-CAP (associated to the Borel parabolic). If $\pi, f$ are as in Proposition \ref{prop:equival} with $\pi$ a P-CAP representation, then $f$ is called a \emph{Saito-Kurokawa} lift.

A non-CAP representation is type I at all primes where $\pi_p$ is spherical, and hence is generic almost everywhere. A CAP representation \emph{is not generic at any place}, and hence cannot have a type I component anywhere. Further, a non-CAP representation is \emph{expected} to satisfy the generalized Ramanujan conjecture, which postulates that it must be tempered at all places and hence conjecturally can never equal one of the representations IIb, IIIb, IVd, Vb/c, Vd or VId.

If $k \ge 3$, and $\pi, f$ are as in Proposition \ref{prop:equival}, then the following additional facts are known to be true:
\begin{itemize}
 \item If $\pi$ is CAP then $\pi$ must be P-CAP (see Corollary 4.5 of \cite{pitschram}) and therefore $\pi_p$ is type IIb whenever it is spherical and type VIb whenever it is not spherical. In particular, $f$ is a Saito-Kurokawa lift.
 \item If $\pi$ is non-CAP, then $\pi_p$ is \emph{tempered} type I whenever it is spherical. This is due to Weissauer \cite{weissram}.
\end{itemize}

We let $S_k(\Gamma_0(N))^{\rm  CAP}$, (resp.\ $S_k(\Gamma_0(N))^{\rm T}$) denote the subspace spanned by all the $f$ as in Proposition \ref{prop:equival} for which the associated $\pi$ are CAP (resp.\ non-CAP). The letter $\rm T$ is chosen because the non-CAP forms are (conjecturally) precisely the ones that are tempered everywhere. Recall that if $k \ge 3$, then it is known that the space $S_k(\Gamma_0(N))^{\rm  CAP}$ is spanned precisely by the Saito-Kurokawa lifts, and that the representations attached to the space $S_k(\Gamma_0(N))^{\rm  T}$  are tempered at all unramified places.

It is clear that we have orthogonal (with respect to Petersson inner product) direct sum decompositions
$$
 S_k(\Gamma_0(N)) =S_k(\Gamma_0(N))^{\rm  CAP}  \oplus S_k(\Gamma_0(N))^{ \rm T},$$ $$S_k(\Gamma_0(N))^{\rm old} =S_k(\Gamma_0(N))^{\rm old, CAP}  \oplus S_k(\Gamma_0(N))^{\rm {old}, \rm T},$$ $$S_k(\Gamma_0(N))^{\rm new} =S_k(\Gamma_0(N))^{\rm new, CAP}  \oplus S_k(\Gamma_0(N))^{\rm {new}, \rm T},$$ $$S_k(\Gamma_0(N))^{\rm CAP} =S_k(\Gamma_0(N))^{\rm new, CAP}  \oplus S_k(\Gamma_0(N))^{\rm old, CAP},$$ $$S_k(\Gamma_0(N))^{\rm T} =S_k(\Gamma_0(N))^{\rm new, T} \oplus S_k(\Gamma_0(N))^{\rm old, T}.
$$
Recall that Conjecture \ref{c:liu} only applies to those $\pi$ that are generic almost everywhere, i.e., only to non-CAP $\pi$. However the CAP case is actually much easier, and Qiu \cite{qiu} has recently proved a theorem that asserts that an analog of Conjecture \ref{c:liu} holds for all CAP representations; see Theorem \ref{thm:qiu}.

\subsection{Newforms and orthogonal Hecke bases}

Let $N$ be a squarefree integer. We say that a non-zero $f $ in $S_k(\Gamma_0(N))$ is a \textbf{newform} if
\begin{enumerate}
 \item $f$ belongs to the newspace $S_k(\Gamma_0(N))^{\rm new}$.
 \item For each prime $p|N$, $f$ is an eigenfunction of the $U(p)$ operator (see \cite[Sect.\ 2.3]{sahaschmidt} for the definition).
 \item The adelization $\phi$ of $f$ generates an irreducible representation $\pi$ (in other words, $\pi$, $f$ are in the situation of Proposition \ref{prop:equival}).
\end{enumerate}

The conditions imply that the automorphic form $\phi$ corresponds to a \emph{factorizable} vector $\phi = \otimes_v \phi_v$ in $\pi$. Indeed, for any $p|N$, the local representation $\pi_p$ has a unique $P_{1,p}(N)$-invariant vector, except if $\pi_p$ is of type IIIa, in which case the space is two dimensional but has (up to multiples) exactly two vectors that are eigenfunctions of the local analog of $U(p)$, which is the  $T_{1,0}$ operator considered earlier.

Using the result from \cite{NPS}, it can be shown that if $f \in S_k(\Gamma_0(N))$ is an eigenfunction of the local Hecke algebras at all places, then it is a newform. However the converse is not true (newforms are not necessarily an eigenfunction of the local Hecke algebra at all primes, in fact they fail to be so precisely at the primes $p$ where the corresponding $\pi_p$ is a type IIIa representation).

\begin{lemma}\label{lemmabasicc}
The space $S_k(\Gamma_0(N))^{\rm new}$ has an orthogonal basis consisting of newforms.
\end{lemma}
\begin{proof} This is immediate  by decomposing the space $S_k(\Gamma_0(N))^{\rm new}$ into mutually orthogonal subspaces corresponding to irreducible automorphic representations, and then using the fact that two linearly independent local vectors in the type IIIa representation can be chosen to be orthogonal.
\end{proof}

Next, for any four positive, mutually coprime,  squarefree integers $a,b,c,d$, we will define a linear map $\delta_{a,b,c,d}$. This map will take $S_k(\Gamma_0(e))^{\rm T}$ to $S_k(\Gamma_0(abcde))^{\rm T}$ for each positive squarefree integer $e$ coprime to $abcd$. It is defined as follows.

Let $f \in S_k(\Gamma_0(e))^{\rm T}$ be such that its adelization $\phi$ generates an irreducible representation $\pi$. (It suffices to define the map on such $f$ because these span the full space $S_k(\Gamma_0(e))^{\rm T}$.)  The automorphic form $\phi$ corresponding to $f$ factors as $\phi = \phi_S \otimes_{p \nmid e} \phi_p$ where $S$ denotes the set of places dividing $e$. Let $p$ be any prime dividing $abcd$. Then the local vector $\phi_p$ is a spherical vector in a type I representation. Using the notation of Sect.~\ref{matcoeffsec}, we have an orthogonal decomposition $\phi_p = \phi_{1,p} + \phi_{2,p} + \phi_{3,p} +\phi_{4,p}$.  Define $\delta_{a,b,c,d}(\phi) = \phi_S \otimes_{p \nmid e} \phi'_p$ where $\phi'_p = \phi_p$ if $p\nmid abcd$, $\phi'_p = \phi_{1,p}$ if $p|a$, $\phi'_p = \phi_{2,p}$ if $p|b$, $\phi'_p = \phi_{3,p}$ if $p|c$, $\phi'_p = \phi_{4,p}$ if $p|d$. Using Proposition \ref{prop:equival}, we see that this takes $f$ to an element $\delta_{a,b,c,d}(f)$ of $S_k(\Gamma_0(abcde))^{\rm T}$.  Note that $\delta_{1,1,1,1}$ is just the identity map.

\begin{proposition}
 For any five coprime squarefree positive integers $a,b,c,d,e$, the map $\delta_{a,b,c,d}$ is an injective linear map that takes $S_k(\Gamma_0(e))^{\rm T}$ to $S_k(\Gamma_0(abcde))^{\rm T}$.  Furthermore, for any positive squarefree integer $N$, we have an orthogonal direct sum decomposition
 $$
  S_k(\Gamma_0(N))^{\rm T}= \bigoplus_{\substack{a,b,c,d,e \\ abcde=N}} \delta_{a,b,c,d} \left(S_k(\Gamma_0(e))^{\rm new, T}\right).
 $$
\end{proposition}
\begin{proof}
By construction, for each $a,b,c,d,e$, the space $\delta_{a,b,c,d} \left(S_k(\Gamma_0(e))^{\rm new, T}\right)$ is the span of all those $f$ in $S_k(\Gamma_0(N))^{\rm T}$ that generate an irreducible representation $\pi_f = \otimes'_v \pi_v$ with the property that $\pi_p$ is spherical if and only if $p \nmid e$, and moreover, $\phi_p = \phi_{1,p}$ for all primes dividing $a$, $\ldots$, $\phi_p =  \phi_{4,p}$ for all primes dividing $d$. Clearly, as $a,b,c,d,e$ vary, this takes care of all representations $\pi = \pi_f$ that come from Proposition \ref{prop:equival}. The orthogonality and injectivity properties come from the corresponding properties of the local maps which were proved earlier. \end{proof}

\begin{corollary}\label{cor:heckebasis}
Let $N>0$ be squarefree. For all positive integers $e|N$, let $\B_{k,e}^{\rm new, T}$ be an orthogonal basis of  $S_k(\Gamma_0(e))^{\rm new, T}$  consisting of newforms (such a basis is known to exist by Lemma \ref{lemmabasicc}). Then the set $ \left\{\delta_{a,b,c,d} (f) \right \}_{\substack{a,b,c,d,e\\ abcde =N \\ f \in \B_{k,e}^{\rm new, T}}}$ gives an orthogonal basis for $S_k(\Gamma_0(N))^{\rm T}$.
\end{corollary}
\begin{proof}
This follows immediately from the above proposition.
\end{proof}

Thus, we have constructed a nice orthogonal Hecke basis for $S_k(\Gamma_0(N))^{\rm  T}$. The adelization $\phi$ of any element of our basis generates an irreducible representation. Moreover, $\phi$ is a factorizable vector and its local components $\phi_p$ are \emph{precisely} the local vectors for which we have computed $J(\phi_p)$ previously.
\subsection{Bessel periods and Fourier coefficients}
Let $f \in S_k(\Gamma_0(N))$ have the Fourier expansion as in \eqref{e:matrixform},
$$
 f(Z) = \sum_{T}a(f,T)  e^{2 \pi i\,{\rm Tr}(TZ)}.
$$
Let $\phi_f$ be the function associated to $f$ via strong approximation, as in~\cite{sch} or \cite{sahapet}. Let $d<0$ be a fundamental discriminant and put
\begin{equation}\label{e:defS}
S=S(d)=  \begin{cases} \left[\begin{smallmatrix}
  \frac{-d}{4} & 0\\
 0 & 1\\\end{smallmatrix}\right] & \text{ if } d\equiv 0\pmod{4}, \\[2ex]
 \left[\begin{smallmatrix} \frac{1-d}{4} & \frac12\\\frac12 & 1\\
 \end{smallmatrix}\right] & \text{ if } d\equiv 1\pmod{4}.\end{cases}
\end{equation}
Given $S$ as above, let the group $T_S$ be defined as before.  So $T_S\simeq K^\times$ where $K=\Q(\sqrt{d})$. Define
$$
 \Cl_K =  T_S(\A) /T_S(\Q)T_S(\R) K_0, \qquad \text{ where } K_0 =  \prod_{p<\infty} (T_S(\Q_p) \cap\GL(2,\Z_p)).
$$
Then $\Cl_K$ can be naturally identified with the ideal class group of $K$ (see Proposition VI.1.3 of \cite{Neu}). We set $h_K = |\Cl_K|$. Let $t_c$, $c\in \Cl_K$, be coset representatives of $\Cl_K$ with $t_c \in \prod_{p<\infty} T_S(\Q_p)$. We can write $t_c = \gamma_{c}m_{c}\kappa_{c}$ with $\gamma_{c} \in \GL(2,\Q)$, $m_{c} \in \GL(2,\R)^+$, and $\kappa_{c}\in\prod_{p<\infty} \GL(2,\Z_p)$; note that $(\gamma_c)_\infty = m_c^{-1}$. The matrices $S_{c} = \det(\gamma_{c})^{-1}\ (\T{\gamma_{c}})S\gamma_{c}$ have discriminant $d$, and form a set of representatives of the $\SL(2,\Z)$-equivalence classes of primitive semi-integral positive definite matrices of discriminant $d$.

Let $\psi:\Q\backslash \A \rightarrow \C^\times$ be the character such that $\psi(x) = e^{2 \pi i x}$ if $x \in \R$ and $\psi(x) = 1$ for $x \in \Z_p$. One obtains a character $\theta_{S}$ of $N(\Q) \backslash N(\A)$ by $\theta_{S}(\mat{1}{X}{}{1}) = \psi({\rm Tr}(S X))$. Let $\Lambda$ be a character of $\Cl_K$. Then, as before, we can define the Bessel period
\begin{equation}\label{defbesselnew}
  B(\phi_f, \Lambda) =
  \int\limits_{\A^\times T_S(\Q)\bs T_S(\A)}\;\int\limits_{N(\Q) \bs N(\A)}\phi_f(tn)\Lambda^{-1}(t) \theta_S^{-1}(n)\,dn\,dt.
\end{equation}
Then we have the following result.
\begin{proposition}\label{p:besselfourier}
 We have
 \begin{equation}
  \frac{B(\phi_f, \Lambda)}{\vl(\A^\times T_S(\Q) \bs T_S(\A))} = e^{-2\pi {\trace}(S)} \frac{1}{h_K}\sum_{c \in \Cl_K} \Lambda(c)^{-1}a(f,S_c),
 \end{equation}
  where we take any  measure on $\A^\times T_S(\Q) \bs T_S(\A)$, but normalize the measure on $N(\Q) \bs N(\A)$ to have total volume 1.
\end{proposition}
\begin{proof} For a function $\phi$ on $T_S(\A)$, which is invariant under $\A^\times T_S(\Q)K_0$, we have
\begin{align*}
\int\limits_{\A^\times T_S(\Q)\bs T_S(\A)} \phi(t)\,dt &= \sum\limits_{c \in \Cl_K} \,\, \int\limits_{\A^\times T_S(\Q)\bs T_S(\Q)T_S(\R)K_0} \phi(t_ct)\,dt \\
&=  \frac{\vl(\A^\times T_S(\Q)\bs T_S(\Q)T_S(\R)K_0)}{\vl(\R^\times \bs T_S(\R))} \sum\limits_{c \in \Cl_K} \,\,\int\limits_{\R^\times \bs T_S(\R)} \phi(t_ct)\,dt.
\end{align*}
Note that $\vl(\A^\times T_S(\Q) \bs T_S(\A)) = h_K \vl(\A^\times T_S(\Q)\bs T_S(\Q)T_S(\R)K_0)$. Hence
$$B(\phi_f, \Lambda) = \frac{\vl(\A^\times T_S(\Q) \bs T_S(\A))}{h_K} \sum_{c \in \Cl_K}\Lambda(c)^{-1} \left(\frac{1}{\vl(\R^\times \bs T_S(\R))} \int_{\R^\times \bs T_S(\R)}(\phi_f)_S(t_ct_\infty)\,dt_\infty\right),$$ where $(\phi_f)_S$ is as defined in Lemma 4.1 of \cite{sahaschmidt}.

A standard calculation, (see Lemma 4.1 and Prop 4.3 of \cite{sahaschmidt}), gives $$\frac{1}{\vl(\R^\times \bs T_S(\R))} \int_{\R^\times \bs T_S(\R)}(\phi_f)_S(t_ct_\infty)\,dt_\infty = e^{-2 \pi \rm{tr}(S)} a(f, S_c).$$ The result follows.
\end{proof}
\subsection{A translation of Conjecture \ref{c:liu}}

Let $\pi$, $f$ and $\phi=\phi_f$ be as in Proposition~\ref{prop:equival} and $\Lambda$ be a character of $\Cl_K$ as in the previous subsection. We will use the notation
\begin{equation}\label{rfdeflamnew}
 R(f, K, \Lambda) = \sum_{c \in \Cl_K}\Lambda(c)^{-1}a(f,c).
\end{equation}
Here, $a(f,c) = a(f, S_c)$. If Conjecture~\ref{c:liu} holds for the representation $\pi$, then Proposition \ref{prop:equival} and Proposition \ref{p:besselfourier} together imply the following: \emph{Assuming $f \in S_k(\Gamma_0(N))^T$, we have}
\begin{equation}\label{e:refggpnew}\begin{split}
 \frac{|R(f, K, \Lambda)|^2}{\langle f, f \rangle} =& e^{4 \pi \Tr(S)}\left(\frac{\vl(Z(\A)G(\Q) \bs G(\A))}{\vl(\Sp(4,\Z)\bs \H_2)}\right) \cdot \left(\frac{h_K}{\vl(\A^\times T_S(\Q) \bs T_S(\A))}\right)^2 \\& \times \frac{C_T}{S_\pi} \cdot \zeta_\Q(2)\zeta_\Q(4) \cdot \frac{L(1/2, \pi \times \AI(\Lambda^{-1}))}{L(1, \pi, \Ad)L(1, \chi_{d})} \cdot J_\infty \cdot \prod_{p|N} J(\phi_p).\end{split}
\end{equation}
The above equality holds for the Tamagawa measure on the global adelic groups. The constant $C_T$ relates the global Tamagawa measure and the product measure used by us. Recall that we defined local measures $dt_p$ on $\Q_p^\times \bs T_S(\Q_p)$ for each non-archimedean place $p$ in Section~\ref{s:localint}. This measure, which was normalized to give total volume 1 on the subgroup $T_S(\Z_p):= T_S(\Q_p)\cap \GL_2(\Z_p)$, was used to compute the local integrals $J_p=J(\phi_p)$. (Note that this measure differs from the measure used in \cite{FM} at the places dividing $d$). Let us now fix some Haar measure on $\R^\times \bs T(\R)$. Then the volume of $\A^\times T(\Q) \bs T(\A)$ with respect to the product measure equals $2h_K \vl(\R^\times \bs T(\R))/w(K)$, a fact that follows from our definition of $\Cl_K$. This gives us $$C_T = \frac{w(K)}{h_K \vl(\R^\times \bs T(\R))}.$$

\begin{proposition}Let $\pi$, $f$ be as in Proposition \ref{prop:equival} and suppose that $f \in S_k(\Gamma_0(N))^T$. Let $\Lambda$ be a character of $\Cl_K$ and suppose that \eqref{e:refggp} holds for our setup.
Then \begin{equation}\label{e:refggpnewer}\begin{split}
 \frac{|R(f, K, \Lambda)|^2}{\langle f, f \rangle} =& e^{4 \pi \Tr(S)}\left(\frac{|d|^{1/2} w(K)^2}{2^3 \ }\right) \cdot \frac{1}{S_\pi}  \cdot \frac{L(1/2, \pi \times \AI(\Lambda^{-1}))}{L(1, \pi, \Ad)} \cdot \frac{J_\infty}{\vl(\R^\times \bs T_S(\R))} \cdot \prod_{p|N} J(\phi_p).\end{split}
\end{equation}
\end{proposition}
\begin{proof}
With the choice of global Tamagawa measures, we get
$$
 \vl(Z(\A)G(\Q) \bs G(\A)) = \vl(\A^\times T(\Q) \bs T(\A)) = 2,
$$
and with our choice of measure $(\det Y)^{-3} dX dY$ on $\H_2$, we have $\vl(\Sp(4,\Z)\bs \H_2) = 2 \pi^{-3} \zeta(2)\zeta(4)= 2 \zeta_\Q(2)\zeta_\Q(4)$.  Finally the class number formula gives $2h_K = w(K) |d|^{1/2} L(1,  \chi_d)$. Substituting these into \eqref{e:refggpnew}, we get \eqref{e:refggpnewer}.
\end{proof}
\subsection{The computation of \texorpdfstring{$J_\infty$}{}}\label{s:arch}

In this section, we compute the archimedean factor $\frac{J_\infty}{\vl(\R^\times \bs T_S(\R))}$ (which does not depend on the choice of measure on $\R^\times \bs T_S(\R)$). We briefly recall the setup. Let $\pi_\infty$ be the discrete series representation of $\PGSp_4(\R)$ with scalar minimal $K$-type $(k,k)$. This is where we need to assume $k>2$ since the case $k=2$ is not discrete series. Let $v$ be a vector in the space of $\pi_\infty$ spanning this one-dimensional $K$-type, and let
$$
 \Phi(g)=\frac{\langle\pi_\infty(g) v,v\rangle}{\langle v,v\rangle}
$$
be the corresponding matrix coefficient. By the formula given in Proposition A.1 of \cite{knightlyli2016}, for $g = \mat{A}{B}{C}{D}$,
$$
 \Phi(g)=\begin{cases}
          \frac{\lambda(g)^k4^k}{\det(A+D+i(C-B))^k}&\text{if }\lambda(g)>0,\\[2ex]
          0&\text{if }\lambda(g)<0.
         \end{cases}
$$
Let the matrix $S = S(d)$ be given by \eqref{e:defS}. Define
\begin{align*}
 J_0&=\int\limits_{\R^\times\backslash T_S(\R)}\int\limits_{M_2^{\rm sym}(\R)}\Phi(t\mat{1}{X}{}{1})e^{-2\pi i\,{\rm Tr}(SX)}\,dX\,dt\
\end{align*}
The quantity $J_\infty$ is given by
\begin{equation}\label{Jinfdef1}J_\infty = \frac{L(1, \pi_\infty, \Ad)L(1, \chi_{d,\infty}) }{\Gamma_\R(2)\Gamma_\R(4)L(1/2, \pi_\infty \otimes \AI(1))} J_0= \frac{\Gamma(2k-1)\pi}{(2\pi)^{2k}}J_0,\end{equation}
  where we use the archimedean $L$-factors given explicitly in Remark \ref{rem:archfactors}.
We now compute $J_0$. We have
$
 T_S(\R)=\R_{>0}\times T_S^1(\R),
$
where $T_S^1(\R)=T_S(\R)\cap\SL_2(\R)=\{g\in \SL_2(\R):\:^tgSg=S\}$. (This uses the fact that all elements in $T_S(\R)$ have positive determinant in our case.) Using the formula for $\Phi(g)$ we have
\begin{align}\label{archcalceq1}
 J_0&=\int\limits_{\{\pm1\}\backslash T_S^1(\R)}\int\limits_{M_2^{\rm sym}(\R)}\Phi(t\mat{1}{X}{}{1})e^{-2\pi i\,{\rm Tr}(SX)}\,dX\,dt\nonumber\\
  &=\int\limits_{\{\pm1\}\backslash T_S^1(\R)}\int\limits_{M_2^{\rm sym}(\R)}\frac{4^k}{\det(A+\,^t\!A^{-1}-iAX)^k}\,e^{-2\pi i\,{\rm Tr}(SX)}\,dX\,dA.
\end{align}

\subsubsection*{Special case \texorpdfstring{$S=1$}{}}
Assume that $d=-4$, so that $S=1$. The computation of $J_0$ in this special case was kindly supplied to us by Kazuki Morimoto and we thank him for allowing us to include this here. In this case $$\R^\times\backslash T_S(\R) \simeq \{\pm1\}\backslash\SO(2)$$ and \eqref{archcalceq1} becomes
\begin{align}\label{archcalceq2}
 J_0&=\int\limits_{\{\pm1\}\backslash\SO(2)}\int\limits_{M_2^{\rm sym}(\R)}\frac{4^k}{\det(A+\,^t\!A^{-1}-iAX)^k}\,e^{-2\pi i\,{\rm Tr}(X)}\,dX\,dA\\
 &=\int\limits_{\{\pm1\}\backslash\SO(2)}\int\limits_{M_2^{\rm sym}(\R)}\frac{4^k}{\det(2-iX)^k}\,e^{-2\pi i\,{\rm Tr}(X)}\,dX\,dA\nonumber\\
 &={\rm vol}(\{\pm1\}\backslash\SO(2))\int\limits_{\R^3}\frac{4^k}{((2-ix)(2-iz)+y^2)^k}\,e^{-2\pi i(x+z)}\,dx\,dy\,dz.
\end{align}
By 3.249.1 of \cite{GR},
$$
 \int\limits_0^\infty\frac1{(y^2+a)^k}\,dy=\frac{(2k-3)!!}{2(2k-2)!!}\frac{\pi}{a^{k-1/2}}.
$$
Hence
\begin{equation}\label{archcalceq4}
 J_0={\rm vol}(\{\pm1\}\backslash\SO(2))\frac{(2k-3)!!}{(2k-2)!!}4^k\pi\bigg(\int\limits_\R\frac{1}{(2-ix)^{k-1/2}}\,e^{-2\pi ix}\,dx\bigg)^2.
\end{equation}
By 3.382.7 of \cite{GR},
$$
 \int\limits_\R\frac1{(\beta-ix)^\nu}e^{-ipx}\,dx=\frac{2\pi p^{\nu-1}e^{-\beta p}}{\Gamma(\nu)}\qquad\text{for }p>0.
$$
Hence, with $\beta=2$, $p=2\pi$ and $\nu=k-1/2$,
$$
 \int\limits_\R\frac1{(2-ix)^{k-1/2}}e^{-2\pi ix}\,dx=\frac{(2\pi)^{k-1/2}e^{-4\pi}}{\Gamma(k-1/2)}.
$$
It follows that
\begin{align}\label{archcalceq5}
 J_0&={\rm vol}(\{\pm1\}\backslash\SO(2))\frac{(2k-3)!!}{(2k-2)!!}4^k\pi\frac{(2\pi)^{2k-1}e^{-8\pi}}{\Gamma(k-1/2)^2}.
\\ &={\rm vol}(\{\pm1\}\backslash\SO(2))\,2^{4k-2}\,\frac{(2\pi)^{2k-1}e^{-8\pi}}{\Gamma(2k-1)}.
\end{align}
where we have used standard properties of the Gamma function. Note that
\begin{equation}\label{archcalceq7}
 {\rm vol}(\{\pm1\}\backslash\SO(2))={\rm vol}(\R^\times\backslash T_S(\R)).
\end{equation}
Hence
\begin{equation}\label{archcalceq8}
 J_0={\rm vol}(\R^\times\backslash T_S(\R))\,2^{4k-2}\,\frac{(2\pi)^{2k-1}e^{-8\pi}}{\Gamma(2k-1)}.
\end{equation}

\subsubsection*{The general case}
We return to arbitrary $S$. Let $V$ be a fixed model for our representation $\pi_\infty$. Let $v_0\in V$ be a fixed non-zero vector of weight $(k,k)$. Let $
 \theta_S(\mat{1}{X}{}{1})=e^{2\pi i\,{\rm Tr}(SX)},$ $\mat{1}{X}{}{1}\in N(\R).$ Recall that we are trying to calculate
\begin{align*}
 J_0&=\int\limits_{\{\pm1\}\backslash T_S^1(\R)}\int\limits_{N(\R)}\frac{\langle\pi_\infty(tn)v_0,v_0\rangle}{\langle v_0,v_0\rangle}\theta_S^{-1}(n)\,dn\,dt.
\end{align*}
For $v,w\in V$, we define
\begin{equation}\label{archcalceq10}
 A(v,w)=\int\limits_{\{\pm1\}\backslash\SO(2)}\int\limits_{N(\R)}\frac{\langle\pi_\infty(tn)v,w\rangle}{\langle v_0,v_0\rangle}\,\theta_1(n)^{-1}\,dn\,dt.
\end{equation}
Note that $A$ is linear in the first variable, and anti-linear in the second variable. The value $A(v_0,v_0)$ is independent of the choice of $v_0$ and the normalization of the hermitian inner product. By \eqref{archcalceq8}, we have
\begin{equation}\label{archcalceq11}
 A(v_0,v_0)={\rm vol}(\R^\times\backslash T_1(\R))\,2^{4k-2}\,\frac{(2\pi)^{2k-1}e^{-8\pi}}{\Gamma(2k-1)}.
\end{equation}
\begin{lemma}\label{archcalclemma1}
 Let $t_0\in\SL(2,\R)$ and $\alpha>0$ be such that $\alpha\,^tt_0St_0=1$. Set
 \begin{equation}\label{archcalclemma1eq1}
  M_d=\mat{t_0^{-1}}{}{}{\alpha\,^tt_0}.
 \end{equation}
 Then
 \begin{equation}\label{archcalclemma1eq2}
  J_0=\alpha^3A(\pi_\infty(M_d)v_0,\pi_\infty(M_d)v_0).
 \end{equation}
\end{lemma}
\begin{proof}
It is easy to see that $\SO(2)=t_0^{-1} T^{1}_S(\R)t_0$. We calculate
\begin{align*}
 A(\pi_\infty(M_d)v_0,\pi_\infty(M_d)v_0)&=
 \int\limits_{\{\pm1\}\backslash\SO(2)}
 \int\limits_{N(\R)}\frac{\langle\pi_\infty(tn)
 \pi_\infty(M_d)v_0,\pi_\infty(M_d)v_0\rangle}{\langle v_0,v_0\rangle}\,\theta_1(n)^{-1}\,dn\,dt\\
  &=\int\limits_{\{\pm1\}\backslash\SO(2)}\int\limits_{N(\R)}\frac{\langle\pi_\infty(t_0tt_0^{-1}M_d^{-1}nM_d)v_0,v_0\rangle}{\langle v_0,v_0\rangle}\,\theta_1(n)^{-1}\,dn\,dt\\
 &=\int\limits_{\{\pm1\}\backslash T_S^1(\R)}\int\limits_{N(\R)}\frac{\langle\pi_\infty(tM_d^{-1}nM_d)v_0,v_0\rangle}{\langle v_0,v_0\rangle}\,\theta_1(n)^{-1}\,dn\,dt\\
  &=\int\limits_{\{\pm1\}\backslash T_S^1(\R)}\int\limits_{M^{\rm sym}_2(\R)}\frac{\langle\pi_\infty(t\mat{1}{\alpha t_0X\,^tt_0}{}{1})v_0,v_0\rangle}{\langle v_0,v_0\rangle}\,e^{-2\pi i\,{\rm Tr}(X)}\,dX\,dt\\
 &=\alpha^{-3}\int\limits_{\{\pm1\}\backslash T_S^1(\R)}\int\limits_{M^{\rm sym}_2(\R)}\frac{\langle\pi_\infty(t\mat{1}{X}{}{1})v_0,v_0\rangle}{\langle v_0,v_0\rangle}\,e^{-2\pi i\,{\rm Tr}(SX)}\,dX\,dt\\
 &=\alpha^{-3}J_0.
\end{align*}
This concludes the proof.
\end{proof}

We recall the function $B_0$ defined in (60) of \cite{PS1} (where we set $m=0$ and $l=l'=k$).  The function $B_0$ is the weight $(k,k)$-vector in the special $(1,\theta_1)$-Bessel model $\mathcal{B}_{1,\theta_1}(\pi_\infty)$ of $\pi_\infty$, normalized such that
\begin{equation}\label{archcalceq12}
 B_0(1)=e^{-4\pi}.
\end{equation} On certain diagonal elements, this function has the formula \begin{equation}\label{holdiscserminKBesselthmeq2}\renewcommand{\arraystretch}{1.4}
  B_0(h(\lambda,\zeta,0,0))=\left\{\begin{array}{l@{\qquad\text{if }}l}
  \lambda^{k}
  e^{-2\pi\lambda(\zeta^2+\zeta^{-2})}&\lambda>0,\\
  0&\lambda<0,\end{array}\right.
 \end{equation}  where we denote
 \begin{equation}\label{helementdefeq}
 h(\lambda,\zeta,0,0):=
 \begin{bmatrix}\lambda \mat{\zeta}{}{}{\zeta^{-1}}&\\
  &\mat{\zeta^{-1}}{}{}{\zeta}\end{bmatrix}.
\end{equation}
We define a Bessel functional $\ell:V\to\C$ as the composition
$$
 V\stackrel{\sim}{\longrightarrow}\mathcal{B}_{1,\theta_1}(\pi_\infty)\longrightarrow\C,
$$
where the first map identifies $v_0$ with $B_0$, and the second map is given by evaluation at $1$. Note that
\begin{equation}\label{archcalceq13}
 \ell(v_0)=B_0(1)=e^{-4\pi}\neq0.
\end{equation}

\begin{lemma}\label{archcalclemma2}
 We have
 \begin{equation}\label{archcalclemma2eq1}
  A(v,w)=c\,\ell(v)\overline{\ell(w)}
 \end{equation}
 for all $v,w\in V$, where
 \begin{equation}\label{archcalclemma2eq3}
  c={\rm vol}(\R^\times\backslash T_S(\R))\,2^{4k-2}\,\frac{(2\pi)^{2k-1}}{\Gamma(2k-1)}.
 \end{equation}
\end{lemma}
\begin{proof}
It is easy to see that for fixed $w\in V$ the map $v\mapsto A(v,w)$ is a $(1,\theta_1)$-Bessel functional on $V$. Hence there exists a constant $c_w$ such that
$$
 A(v,w)=c_w\ell(v)\qquad\text{for all }v\in V.
$$
It is further easy to see that the map $w\mapsto \overline{c_w}$ is also a $(1,\theta_1)$-Bessel functional on $V$. Hence there exists a constant $c$ such that
$$
 \overline{c_w}=c\ell(w)\qquad\text{for all }w\in V.
$$
Thus we get
\begin{equation}\label{archcalclemma2eq2}
  A(v,w)=c\,\ell(v)\overline{\ell(w)}
\end{equation}
for all $v,w\in V$.
Substituting $v=w=v_0$ and observing \eqref{archcalceq11} and \eqref{archcalceq13}, we see that
$$
 c={\rm vol}(\R^\times\backslash T_S(\R))\,2^{4k-2}\,\frac{(2\pi)^{2k-1}}{\Gamma(2k-1)}.
$$
This concludes the proof.
\end{proof}

From Lemma \ref{archcalclemma2} and \eqref{archcalclemma1eq2} we now get
\begin{align}\label{archcalceq14}
 J_0&=\alpha^3A(\pi_\infty(M_d)v_0,\pi_\infty(M_d)v_0)\nonumber\\
  &=\alpha^3c\,|B_0(M_d)|^2.
\end{align}
First assume that $S=\mat{\frac{|d|}4}{}{}{1}$. Then we may take
$$
 t_0=\mat{2^{1/2}|d|^{-1/4}}{}{}{2^{-1/2}|d|^{1/4}},\qquad\alpha=2|d|^{-1/2},
$$
so that
\begin{align*}
 M_d&=\begin{bmatrix}2^{-1/2}|d|^{1/4}\\&2^{1/2}|d|^{-1/4}\\&&\alpha\mat{2^{1/2}|d|^{-1/4}}{}{}{2^{-1/2}|d|^{1/4}}\end{bmatrix}
 \\&=\alpha\begin{bmatrix}\alpha^{-1}\mat{2^{-1/2}|d|^{1/4}}{}{}{2^{1/2}|d|^{-1/4}}\\&&\mat{2^{1/2}|d|^{-1/4}}{}{}{2^{-1/2}|d|^{1/4}}\end{bmatrix}\\
  &=\alpha h(2^{-1}|d|^{1/2},2^{-1/2}|d|^{1/4},0,0).
\end{align*}
By \eqref{holdiscserminKBesselthmeq2},
$
 B_0(h(\lambda,\zeta,0,0))=\lambda^ke^{-2\pi\lambda(\zeta^2+\zeta^{-2})},
$
which gives
$$B_0(M_d)=\Big(\frac{|d|}4\Big)^{k/2}e^{-2\pi(|d|/4+1)}.
$$
Substituting into \eqref{archcalceq14}, we get
\begin{equation}\label{archcalceq15}
 J_0={\rm vol}(\R^\times\backslash T_S(\R))\,2^{4k-2}\,\frac{(2\pi)^{2k-1}}{\Gamma(2k-1)}\,\Big(\frac{|d|}4\Big)^{k-3/2}e^{-4\pi(|d|/4+1)}.
\end{equation}
Now assume that $S=\mat{\frac{1-d}{4}}{\frac12}{\frac12}{1}$. Then we may take
$$
 t_0=\mat{1}{}{-\frac12\;}{1}\mat{2^{1/2}|d|^{-1/4}}{}{}{2^{-1/2}|d|^{1/4}},\qquad\alpha=2|d|^{-1/2},
$$
so that
$$
 M_d=\begin{bmatrix}\mat{2^{-1/2}|d|^{1/4}}{}{}{2^{1/2}|d|^{-1/4}}\mat{1}{}{\frac12\;}{1}\\&&\alpha\mat{2^{1/2}|d|^{-1/4}}{}{}{2^{-1/2}|d|^{1/4}}\mat{1}{\;-\frac12}{}{1}\end{bmatrix}.
$$
Lemma 5.1 of \cite{KP} shows that
$$
 \mat{2^{1/2}|d|^{-1/4}}{}{}{2^{-1/2}|d|^{1/4}}\mat{1}{\;-\frac12}{}{1}=k_1\mat{\zeta^{-1}}{}{}{\zeta}k_2
$$
with $k_1,k_2\in\SO(2)$ and
$$
 \zeta^2=\frac{|d|+5-\sqrt{|d|^2-6|d|+25}}{4|d|^{1/2}}.
$$
Hence
$$
 M_d=\alpha\mat{k_1}{}{}{k_1}\begin{bmatrix}\alpha^{-1}\zeta\\&\alpha^{-1}\zeta^{-1}\\&&\zeta^{-1}\\&&&\zeta\end{bmatrix}\mat{k_2}{}{}{k_2}.
$$
In this case, a similar calculation to before gives
$$
 B_0(M_d)=\Big(\frac{|d|}4\Big)^{k/2}e^{-2\pi\frac{|d|+5}{4}}.
$$
Substituting into \eqref{archcalceq14} gives
\begin{equation}\label{archcalceq16}
 J_0={\rm vol}(\R^\times\backslash T_S(\R))\,2^{4k-2}\,\frac{(2\pi)^{2k-1}}{\Gamma(2k-1)}\,\Big(\frac{|d|}4\Big)^{k-3/2}e^{-4\pi\frac{|d|+5}{4}}.
\end{equation}
Summarizing \eqref{archcalceq15} and \eqref{archcalceq16}, we get
\begin{equation}\label{archcalceq17}
 J_0={\rm vol}(\R^\times\backslash T_S(\R))\,2^{4k-2}\,\frac{(2\pi)^{2k-1}}{\Gamma(2k-1)}\,\Big(\frac{|d|}4\Big)^{k-3/2}e^{-4\pi{\rm Tr}(S)}.
\end{equation}
As a consequence, for any $k>2$,
\begin{align}\label{archcalceq18}
 J_\infty&=\frac{\Gamma(2k-1)\pi}{(2\pi)^{2k}}J_0\nonumber\\
 &={\rm vol}(\R^\times\backslash T_S(\R))\,2^{4k-3}\,\Big(\frac{|d|}4\Big)^{k-3/2}e^{-4\pi\,{\rm Tr}(S)}.
\end{align}
This formula will be used in the proof of Theorem \ref{t:main} below.
\begin{remark} Hsieh and Yamana \cite[Prop. 5.7]{hsieh-yamana} have recently computed the factor $J_\infty$  using a different method based on Shimura's work on confluent hypergeometric functions.
\end{remark}
\subsection{The main result and some consequences}\label{s:conseq}
Throughout this subsection, $N$ denotes a positive squarefree odd integer and $k > 2$ an integer.
We begin by stating the main theorem of this paper.
\begin{theorem}\label{t:main}
 Let $f \in S_k(\Gamma_0(N))^{\rm T}$. Assume that $f=\delta_{a,b,c,d} (g)$ where $abcde =N$ and $g$ is a newform in $S_k(\Gamma_0(e))^{\rm T}$; note that $f$ itself is a newform if and only if $a=b=c=d=1$ (whence $f=g$). Let $\pi$ be the representation attached to $g$ (or equivalently, to $f$).  Let $d<0$ be a fundamental discriminant such that $\big(\frac{d}{p}\big)=-1$ for all $p|N$ and let $\Lambda$ be an ideal class character of $K=\Q(\sqrt{d})$. Furthermore, assume that at least one of the following statements is true:
 \begin{enumerate}
 \item Conjecture \ref{c:liu} is true for $\pi$,
 \item $\Lambda =1$,
 \item $f$ is a weak Yoshida lift (in the sense of \cite{sahapet}).
 \end{enumerate}Then we have
 \begin{align*}
   \frac{|R(f, K, \Lambda)|^2}{\langle f, f \rangle} &=  2^{2k-s} \cdot w(K)^2 \ |d|^{k-1}\frac{L(1/2, \pi \times \AI(\Lambda^{-1}))}{ \ L(1, \pi, \Ad)} \prod_{p|N} J_p\\ &=  \frac{2^{4k-s} \cdot \pi^{2k+1}}{(2k-2)!} \ w(K)^2 \ |d|^{k-1}\frac{L_\f(1/2, \pi \times \AI(\Lambda^{-1}))}{ \ L_\f(1, \pi, \Ad)}\prod_{p|N} J_p,
 \end{align*}
 where $s = 5$ if $f$ is a weak Yoshida lift and $s=4$ otherwise. The quantities $J_p$ are given as follows:
 $$
   J_p = \begin{cases} L(1, \pi_p, \mathrm{Std})(1-p^{-4}) &\text{ if }p|bd, \text{ (cannot occur if $f$ is a newform),} \\  L(1, \pi_p, \mathrm{Std})(1-p^{-4}) p^{-1} &\text{ if }p|ac, \text{ (cannot occur if $f$ is a newform),} \\ (1+p^{-2})(1+p^{-1}) & \text{ if } p|e \text{ and } \pi_p \text{ is of type IIIa}, \\ 2(1+p^{-2})(1+p^{-1}) & \text{ if } p|e \text{ and } \pi_p \text{ is of type VIb}, \\ 0 & \text{ otherwise.} \end{cases}
 $$
\end{theorem}

\begin{proof} This follows by combining \eqref{e:refggpnewer} with Theorem \ref{t:localunram}, Theorem \ref{t:localintmain} and \eqref{archcalceq18}. Note that $S_\pi = 4$ if $f$ is a weak Yoshida lift and $S_\pi=2$ otherwise. We note that Conjecture \ref{c:liu} is currently known in the cases where $\Lambda =1$ or $f$ is a weak Yoshida lift.
\end{proof}

Theorem \ref{t:mainclassi} is valid for forms that are not Saito-Kurokawa lifts. We now prove an unconditional theorem for Saito-Kurokawa lifts. This relies on a result due to Qiu \cite{qiu}. We quote Qiu's result in a very special case:

\medskip

\begin{theorem}[Qiu] \label{thm:qiu} \emph{Let $\pi$, $f$, $\phi$ be as in Propostion~\ref{prop:equival}. Assume further that $f$ is a Saito-Kurokawa lift (i.e., the corresponding representation $\pi$ is of P-CAP type) and that $\phi = \otimes_v \phi_v$ corresponds to a factorizable vector.  Let $\pi_0$ be the automorphic representation of $\GL(2,\A)$ that $\pi_f$ is associated to, i.e., $L_\f(s, \pi_f) = L_\f(s, \pi_0)\zeta(s+1/2)\zeta(s-1/2).$ Let $d<0$ be a fundamental discriminant, $K = \Q(\sqrt{d})$, and $\Lambda$ an ideal class character of $K$. Then $R(f, K, \Lambda) = 0$ if $\Lambda$ is a non-trivial character, and
 \begin{equation}\label{e:refggpnewSK}
  \frac{|R(f, K, 1)|^2}{\langle f, f \rangle} = \left(\frac{3|d| w(K)^2}{4 \ }\right) e^{4\pi {\rm Tr}(S))}\frac{L(1/2, \pi_0 \otimes \chi_d)L(1, \chi_d)^2}{L(3/2, \pi_0)L(1, \pi_{0}, \Ad)} \cdot \frac{J_\infty}{\vl(\R^\times \bs T_S(\R))}  \cdot \prod_{p|N}\left(2^{-1} J^*(\phi_p) \right).
 \end{equation}}
\end{theorem}

\begin{remark}Equation \eqref{e:refggpnewSK} follows from the \textit{first} formula in Theorem 2 of \cite{qiu}.  The second equality stated therein is not correct: in the case $F$ is totally real of degree $d$, one must multiply the right hand side of \cite[(4.13)]{qiu} by $(2\pi)^d$, as can be checked by computing $L(0, \Pi_v, \chi_v)$ for $v$ a real place.  Of course, the product over $v$ should also be of $2^{-1} \mathcal{P}_v^{\#}(f_{1, v}, f_{2, v})$.\end{remark}

Now using \eqref{e:refggpnewSK} instead of \eqref{e:refggpnewer} and the same argument as before, we get
\begin{theorem}\label{t:mainskformula}
 Let $f \in S_k(\Gamma_0(N))^{\rm new, CAP}$. Assume that $f$ is a newform and  is a Saito-Kurokawa lift.  Let $\pi_0$ be the automorphic representation of $\GL(2,\A)$ that $\pi_f$ is associated to, i.e., $L_\f(s, \pi_f) = L_\f(s, \pi_0)\zeta(s+1/2)\zeta(s-1/2).$ Let $d<0$ be a fundamental discriminant such that $\big(\frac{d}{p}\big)=-1$ for all $p|N$ and let $\Lambda$ be an ideal class character of $K=\Q(\sqrt{d})$. Then $R(f, K, \Lambda) = 0$ if $\Lambda$ is non-trivial, and
 $$
  \frac{|R(f, K, 1)|^2}{\langle f, f \rangle}=(3 \cdot 2^{2k-2} ) w(K)^2 \ |d|^{k-\frac12}\frac{L(1/2, \pi_0 \times \chi_d)L(1, \chi_d)^2}{ L(3/2, \pi_0) L(1, \pi_0, \Ad)} \prod_{p|N} (1+p^{-2})(1+p^{-1}).
 $$
\end{theorem}

\subsubsection*{A sanity check} Consider the special case $N=1$.  Let $f$ be a Saito-Kurokawa lift of weight $k$ as above. Let $g$ be the classical cusp form of half-integral weight $k-1/2$ that gives rise to $f$. We have the following classical results:
\begin{equation} |R(f, K, 1)|^2 = |d|\left(\frac{w(K)L(1, \chi_d)}{2}\right)^2 |c(g,d)|^2, \end{equation}
 which can be derived by combining the equation on page 74 of \cite{eichzag} with the class number formula.  Also, from \cite{brown07}, Corollary 6.3,\begin{footnote}{Note that the formula stated in \cite{brown07} differs from ours by a factor of $6 = [\SL(2,\Z):\Gamma_0(4)]$.  This is because the Kohnen--Zagier formula is quoted incorrectly in Lemma 6.1 of \cite{brown07}, owing to a different normalization of the Petersson inner product.}\end{footnote}
\begin{equation}  \frac{|c(g,d)|^2}{\langle f, f\rangle} = \frac{3\pi}{k-1} |d|^{k-3/2}2^{2k+1} \frac{L_\f(1/2, \pi_0 \otimes \chi_d)}{L_\f(3/2, \pi_0)L(1, \pi_0, \Ad)}. \end{equation} Combining these formulas, we recover Theorem \ref{t:mainskformula} exactly for $N=1$. This gives an independent verification of all our computations.

\medskip

We now work out several consequences of Theorem \ref{t:main}. For each $f$ as in Theorem \ref{t:main}, let us use the notation \begin{equation}\label{defjfn}J_{f,N} = \prod_{p|N} J_p.\end{equation}
We have the following bound, which follows immediately from the formulas for $J_p$,
\begin{equation}\label{Jfnbd}
 J_{f,N} \ll_\epsilon N^\epsilon.
\end{equation}
\subsubsection*{Yoshida lifts}
We recall the representation theoretic construction of scalar valued Yoshida lifts described in \cite{sahaschmidt}. Let $N_1$, $N_2$ be two positive, squarefree integers such that $M = \gcd(N_1, N_2)>1$. Let $g_1$ be a classical newform of weight $2k-2$ and level $N_1$ and $g_2$ be a classical newform of weight $2$ and level $N_2$, such that $g_1$ and $g_2$ are not multiples of each other. Assume that for all primes $p$ dividing $M$ the Atkin-Lehner eigenvalues of $g_1$ and $g_2$ coincide. Put $N = \mathrm{lcm}(N_1, N_2)$. For any $p|M$, let $\delta_p$ be the common Atkin-Lehner eigenvalue of $g_1$ and $g_2$. Then for any divisor $M_1$ of $M$ with an \emph{odd} number of prime factors, there exists a non-zero holomorphic Siegel cusp form $f= f_{g_1,g_2;M_1}$ such that $f$ is a newform in $S_k(\Gamma_0(N))^{\rm T}$. Furthermore, the automorphic representation $\pi$ generated by the adelization of $f$ has the following properties.
 \begin{enumerate}

\item $L(s, \pi_f)=L(s, \pi_{g_1})L(s, \pi_{g_2}).$

\item If $p|N$, $p \nmid M$, then $\pi_p$ is of type IIa.

\item If $p|M$, $p \nmid M_1$, then $\pi_p$ is of type VIa. The associated character $\sigma_p$ is trivial if $\delta_p =-1$ and unramified quadratic if $\delta_p=1$.

    \item If $p|M_1$, then $\pi_p$ is of type VIb. The associated character $\sigma_p$ is trivial if $\delta_p =-1$ and unramified quadratic if $\delta_p=1$.

\end{enumerate}

Let $f = f_{g_1,g_2;M_1}$ be a Yoshida lift, as above. Then $S_\pi = 4$. It is of interest to see what Theorem \ref{t:main} gives us in this case. Indeed, in this case Conjecture \ref{c:liu} has been proved already, so we get an \emph{unconditional statement}.

\begin{proposition}\label{propyoshida}
 Let $N_1$, $N_2$, $g_1$, $g_2$, $M$, $M_1$ be as above. Assume that $N_1$ and $N_2$ are odd and $k>2$. Let $f= f_{g_1,g_2;M_1}$ be a Yoshida lift. Then, for any $d<0$ a fundamental discriminant such that $\big(\frac{d}{p}\big)=-1$ for all $p|N$, and any ideal class character $\Lambda$ of $K=\Q(\sqrt{d})$, we have
 \begin{equation}\label{yoshidaeq3}
  \frac{|R(f, K, \Lambda)|^2}{\langle f, f \rangle} =\begin{cases}
   \displaystyle\frac{2^{4k-5} \pi^{2k+1} w(K)^2  |d|^{k-1}}{(2k-2)!} \times \prod_{p|N} \left(2(1+p^{-2})(1+p^{-1})\right) \\[4ex]
   \displaystyle \quad\times\;\frac{L_\f(1/2, \pi_{g_1} \times \AI(\Lambda^{-1}))L_\f(1/2, \pi_{g_2}\times \AI(\Lambda^{-1}))   } {L_\f(1, \pi_{g_1}, \Ad)L_\f(1, \pi_{g_2}, \Ad)L_\f(1, \pi_{g_1}\times \pi_{g_2})}  &\text{ if } N_1 = N_2 = M =M_1, \\[2ex]
   0 & \text{ otherwise. }\end{cases}
 \end{equation}
\end{proposition}
In the special case $\Lambda=1$, the above result can probably also be derived from the formulas in \cite{boschdum}.
\subsubsection*{Averages of $L$-functions}
For each $e|N$, let $\B_{k,e}^{\rm new, T}$ be an orthogonal basis of  $S_k(\Gamma_0(e))^{\rm new, T}$  consisting of newforms. Put
$$
 \B_{k,N}^{\rm T} = \bigcup_{abcde = N} \{\delta_{a,b,c,d} (g): g\in \B_{k,e}^{\rm new, T} \}.
$$
Then by Corollary \ref{cor:heckebasis}, $\B_{k,N}^{\rm T}$ is an orthogonal basis for $S_k(\Gamma_0(N))^{\rm T}$. Each $f \in \B_{k,N}^{\rm T}$ is a Hecke eigenform at places not dividing $N$, and gives rise to an irreducible cuspidal automorphic representation $\pi_f$.
\begin{theorem}\label{t:aver}
 Let $k \ge 6$ be even and $N$ squarefree and odd. Let $\B_{k,N}^{\rm T}$ be as above. Suppose that Conjecture \ref{c:liu} holds.  Fix a fundamental discriminant $d<0$ and an ideal class character $\Lambda$ of $K =\Q(\sqrt{d})$. Suppose that $\left(\frac{d}{p} \right)=-1$ for all $p|N$. Put $l = 1$ if $\Lambda^2=1$ and $l=2$ otherwise. Then
 \begin{equation}\label{e:averagemain}
  \sum_{f\in \B_{k,N}^{\rm T}} \frac{L_\f(1/2, \pi_f \times \AI(\Lambda^{-1}))}{ \ L_\f(1, \pi_f, \Ad) \ } \cdot J_{f,N} \cdot u_f= \frac{k^3 \ L_\f(1, \chi_d) \ [\Sp(4,\Z): \Gamma_0(N)]}{l \ \pi^6} \left(1+ O(k^{-2/3}N^{-1}) \right),
 \end{equation}
 where $J_{f,N}$ is as defined in \eqref{defjfn} and $u_f = 2$ if $f$ is a weak Yoshida lift, and $u_f = 2^2$ otherwise. The implied constant in $O$ depends only on $d$.
\end{theorem}
\begin{proof}
Let
$
 t_{k,d} = \frac{2^{4k-4} \cdot \pi^{2k+1}}{(2k-2)!} \ w(K)^2 \ |d|^{k-1}.
$
Assuming Conjecture \ref{c:liu}, our main theorem above says
$$
 \frac{|R(f, K, \Lambda)|^2}{\langle f, f \rangle} = t_{k,d} \cdot u_f \cdot \frac{L_\f(1/2, \pi_f \times \AI(\Lambda^{-1}))}{ \ L_\f(1, \pi_f, \Ad)}J_{f,N}.
$$
On the other hand, by Theorem 7.3 and Corollary 9.4 of \cite{dickson}, we know that $$\sum_{f\in \B_{k,N}^{\rm T}} \frac{\pi^6 l}{k^3 t_{k,d} \ [\Sp(4,\Z): \Gamma_0(N)]L_\f(1, \chi_d)} \frac{|R(f, K, \Lambda)|^2}{\langle f, f \rangle} = 1 + O_d(N^{-1}k^{-2/3}).$$
 This completes the proof.
  \end{proof}

If we put $N=1$ in \eqref{e:averagemain}, we get the result \eqref{introaverage} stated in the introduction. Furthermore, we can use the above theorem and sieve for various divisors of $N$ to get a version of \eqref{e:averagemain} that only sums over newforms, i.e., elements of $\B_{k,N}^{\rm new, T}$. We state and prove such a result in the simplest case of prime modulus.

\begin{corollary}\label{corr:aver}Let $k \ge 6$ be even, let $p$ be an odd prime and let $\B_{k,p}^{\rm new, T}$ be an orthogonal basis of $S_k(\Gamma_0(p))$ consisting of newforms. Let $\Phi_{k,p}^{\rm new, T}$ be the set of distinct irreducible subspaces\footnote{The ``multiplicity one" conjecture predicts that two such subspaces are distinct if and only if the corresponding automorphic representations are non-isomorphic; however we do not assume the truth of this conjecture here.} of $L^2(Z(\Q)G(\Q)\bs G(\A))$ generated by the adelizations of elements of $\B_{k,p}^{\rm new, T}$.
Suppose that Conjecture \ref{c:liu} holds.  Fix a fundamental discriminant $d<0$ such that $\big(\frac{d}{p} \big)=-1$, and an ideal class character $\Lambda$ of $K =\Q(\sqrt{d})$. Put $l = 1$ if $\Lambda^2=1$ and $l=2$ otherwise. Then

  $$\sum_{\substack{\pi \in \Phi_{k,p}^{\rm new, T}\\ \pi_p \in \{\text{\rm IIIa, VIb}\}}} \frac{L_\f(1/2, \pi \times \AI(\Lambda^{-1}))}{ \ L_\f(1, \pi, \Ad) \ } \ u_\pi= \frac{k^3  p^3 L_\f(1, \chi_d) }{l \ \pi^6} + O_d(k^{7/3}p^2 + k^3),$$  where $u_\pi = 2$ if $\pi$ is a weak endoscopic lift, and $u_\pi = 2^2$ otherwise.
\end{corollary}

\begin{proof}This follows from Theorem \ref{t:aver} for $N=1$ and $N=p$, combined with the observation that for any $\pi \in \Phi_{k,p}^{\rm new, T}$ with $\pi_p$ of type IIIa or VIb, $\sum_{\substack{f \in \B_{k,p}^{\rm new, T} \\ \pi_f = \pi}}J_{f,p} = 2 + O(p^{-1})$; see Remark \ref{rem:jpdim}.
\end{proof}

\subsubsection*{Size of Fourier coefficients}As an interesting consequence of Theorem \ref{t:main}, we \emph{predict} the best possible upper bound for $|R(f,K, \Lambda)|$.

\begin{proposition}\label{rfklambdabd}Let $f$, $\pi$, $d$, $K$, $\Lambda$ be as in Theorem \ref{t:main}. Assume the truth of both Conjecture \ref{c:liu} and the Generalized Riemann Hypothesis for $\pi$. Then

$$|R(f,K, \Lambda)| \ll_{\epsilon} \langle f, f \rangle^{1/2} (2\pi e)^{k} \ k^{-k+\frac{3}{4}} |d|^{\frac{k-1}{2}} (Nkd)^\epsilon.$$

\end{proposition}
\begin{proof} This follows from Theorem \ref{t:main}, the bound \eqref{Jfnbd}, and Stirling's formula. Note that the GRH for $\pi$ implies that the quotient of the various finite parts of $L$-functions appearing in \ref{t:main} is bounded by $(Nkd)^\epsilon$.
\end{proof}

Unfortunately, the above proposition does not seem to predict the strongest expected bound on an \emph{individual} Fourier coefficient $|a(F,S)|$. That would require knowledge of the distribution of the size and arguments of the complex numbers $R(f,K, \Lambda)$ as $\Lambda$ varies over characters of $\Cl_K$, which appears to be a difficult problem. We note here that a conjecture of Resnikoff and Saldana \cite{res-sald} predicts that as $d$ varies, we have $|a(F,S)| \ll_{F} d^{k/2 -3/4 + \epsilon}$ for any $S$ with $\disc(S)=d$. The bound of Proposition \ref{rfklambdabd} and the conjectured bound of Resnikoff and Saldana do not appear to imply each other in any direction without assuming equidistribution of the arguments of the Fourier coefficients.
\subsubsection*{Integrality of $L$-values}
If $L(s, \mathcal{M})$ is an $L$-series associated to an arithmetic object $\mathcal{M}$, it is of interest to study its values at certain critical points $s=m$. For these critical points, conjectures due to Deligne predict that  $L(m,\mathcal{M})$ is the product of a suitable transcendental number $\Omega$ and an algebraic number. One can go even further and try to predict what primes divide the numerator and denominator of the algebraic number above (once the period is suitably normalized).  Our main theorem implies the following result.

\begin{proposition}\label{prop:algebraicity}
 Let $k>2$, $N$ squarefree and odd, and $f \in S_k(\Gamma_0(N))^{\rm T}$ be a newform such that all its Fourier coefficients are algebraic integers.\footnote{Using essentially the same method as Prop.\ 3.18 of \cite{sahapet}, it can be shown that the space $S_k(\Gamma_0(N))^{\rm new, T}$ has an orthogonal basis consisting of newforms whose Fourier coefficients are algebraic integers.} Let $\pi$ be the representation attached to $f$, and assume that $\pi_p$ is of type {\rm IIIa} or {\rm VIb} for all $p|N$.  Assume the truth of Conjecture \ref{c:liu} for $\pi$. Then, for all fundamental discriminants $d<0$ such that $\big(\frac{d}{p} \big)=-1$ for all $p|N$ and all ideal class characters $\Lambda$ of $K=\Q(\sqrt{d})$, we have
$$\langle f, f\rangle \times \frac{L_\f(1/2, \pi_f \times \AI(\Lambda^{-1}))}{ \ L_\f(1, \pi_f, \Ad)} \times \frac{3^2 \cdot 2^{4k}\pi^{2k+1} \ |d|^{k-1} \sigma_0(N) \sigma_1(N) \sigma_2(N)} {\ (2k-2)! \ N^3}$$  is an algebraic integer, where $\sigma_i(N) = \sum_{1 \le d|N}d^i$.
\end{proposition}

\begin{proof} This follows from Theorem \ref{t:main} as our assumptions imply that $R(f, K, \Lambda)$ is an algebraic integer.
\end{proof}

\emph{Proof of Proposition \ref{simularith}.} We first claim that there exists a Yoshida lift $f$ of $g_1$, $g_2$ whose Fourier coefficients are algebraic \emph{integers} in the CM-field generated by $g_1$, $g_2$. Indeed, by Theorem 6 of \cite{sahapet}, it follows that such a lift exists whose Fourier coefficients are algebraic \emph{numbers} in the field generated by $g_1$, $g_2$. Now, the main result of \cite{shimura75} assures us that one can always multiply any Siegel modular form (of any level) with algebraic  Fourier coefficients  with a large  enough integer such that the Fourier coefficients are algebraic integers. This proves the claim. Given such a lift $f$, we now apply Proposition \ref{propyoshida}. We can absorb all the remaining terms appearing in Proposition \ref{propyoshida} that depend only on $f$, into the constant $\Omega$. This immediately the proof.

\bibliography{bocherer_bessel}{}

\begin{thebibliography}{10}

\bibitem{asgsch}
Mahdi Asgari and Ralf Schmidt.
\newblock Siegel modular forms and representations.
\newblock {\em Manuscripta Math.}, 104(2):173--200, 2001.

\bibitem{asgarischmidt08}
Mahdi Asgari and Ralf Schmidt.
\newblock On the adjoint {$L$}-function of the {$p$}-adic {$\rm GSp(4)$}.
\newblock {\em J. Number Theory}, 128(8):2340--2358, 2008.

\bibitem{boch-conj}
Siegfried B{\"o}cherer.
\newblock Bemerkungen {\"u}ber die {D}irichletreihen von {K}oecher und {M}aass.
\newblock {\em Mathematica Gottingensis}, 68:36 pp., 1986.

\bibitem{boschdum}
Siegfried B{\"o}cherer, Dummigan Neil, and Rainer Schulze-Pillot.
\newblock {Y}oshida lifts and {S}elmer groups.
\newblock {\em J. Math. Soc. Japan}, 64(4):1353--1405, 2012.

\bibitem{Borel1976}
Armand Borel.
\newblock Admissible representations of a semi-simple group over a local field
  with vectors fixed under an {I}wahori subgroup.
\newblock {\em Invent. Math.}, 35:233--259, 1976.

\bibitem{brown07}
Jim Brown.
\newblock An inner product relation on {S}aito-{K}urokawa lifts.
\newblock {\em Ramanujan J.}, 14(1):89--105, 2007.

\bibitem{bruinnonvan}
Jan~Hendrik Bruinier.
\newblock Nonvanishing modulo {$l$} of {F}ourier coefficients of half-integral
  weight modular forms.
\newblock {\em Duke Math. J.}, 98(3):595--611, 1999.

\bibitem{Casselman1980}
William Casselman.
\newblock The unramified principal series of {${\mathfrak{p}}$}-adic groups.
  {I}. {T}he spherical function.
\newblock {\em Compositio Math.}, 40(3):387--406, 1980.

\bibitem{corbett}
Andrew Corbett.
\newblock A proof of the refined {G}an-{G}ross-{P}rasad conjecture for
  non-endoscopic {Y}oshida lifts.
\newblock {\em Preprint}, 2015.

\bibitem{dickson}
Martin Dickson.
\newblock Local spectral equidistribution for degree two {S}iegel modular forms
  in level and weight aspects.
\newblock {\em Int. J. Number Theory}, 11(2):341--396, 2015.

\bibitem{eichzag}
Martin Eichler and Don Zagier.
\newblock {\em The theory of {J}acobi forms}, volume~55 of {\em Progress in
  Mathematics}.
\newblock Birkh\"auser Boston Inc., Boston, MA, 1985.

\bibitem{fur}
Masaaki Furusawa.
\newblock On {$L$}-functions for {${\rm GSp}(4)\times {\rm GL}(2)$} and their
  special values.
\newblock {\em J. Reine Angew. Math.}, 438:187--218, 1993.

\bibitem{FM}
Masaaki Furusawa and Kazuki Morimoto.
\newblock Refined global {G}ross-{P}rasad conjecture on special {B}essel
  periods and {B}\"ocherer's conjecture.
\newblock arXiv:1611.05567.

\bibitem{ggp}
Wee~Teck Gan, Benedict Gross, and Dipendra Prasad.
\newblock Symplectic local root numbers, central critical {$L$} values, and
  restriction problems in the representation theory of classical groups.
\newblock {\em Ast\'erisque}, (346):1--109, 2012.
\newblock Sur les conjectures de Gross et Prasad. I.

\bibitem{GR}
I.~S. Gradshteyn and I.~M. Ryzhik.
\newblock {\em Table of integrals, series, and products}.
\newblock Elsevier/Academic Press, Amsterdam, seventh edition, 2007.
\newblock Translated from the Russian, Translation edited and with a preface by
  Alan Jeffrey and Daniel Zwillinger, With one CD-ROM (Windows, Macintosh and
  UNIX).

\bibitem{harris-kudla-1991}
Michael Harris and Stephen Kudla.
\newblock The central critical value of a triple product {$L$}-function.
\newblock {\em Ann. of Math. (2)}, 133(3):605--672, 1991.

\bibitem{hsieh-yamana}
Ming-Lun Hsieh and Shunsuke Yamana.
\newblock Bessel periods and anticyclotomic $p$-adic spinor {$L$}-functions.
\newblock {\em Preprint}.

\bibitem{ichino}
Atsushi Ichino.
\newblock Trilinear forms and the central values of triple product
  {$L$}-functions.
\newblock {\em Duke Math. J.}, 145(2):281--307, 2008.

\bibitem{ichino-ikeda}
Atsushi Ichino and Tamotsu Ikeda.
\newblock On the periods of automorphic forms on special orthogonal groups and
  the {G}ross-{P}rasad conjecture.
\newblock {\em Geom. Funct. Anal.}, 19(5):1378--1425, 2010.

\bibitem{KP}
Rodney Keaton and Ameya Pitale.
\newblock Restrictions of {E}isenstein series and {R}ankin-{S}elberg
  convolution.
\newblock arXiv:1706.09342.

\bibitem{Klingen1990}
Helmut Klingen.
\newblock {\em Introductory lectures on {S}iegel modular forms}, volume~20 of
  {\em Cambridge Studies in Advanced Mathematics}.
\newblock Cambridge University Press, Cambridge, 1990.

\bibitem{knightlyli2016}
A.~{Knightly} and C.~{Li}.
\newblock {On the distribution of Satake parameters for Siegel modular forms}.
\newblock {\em ArXiv e-prints}, May 2016.

\bibitem{kohnwalds}
Winfried Kohnen.
\newblock Fourier coefficients of modular forms of half-integral weight.
\newblock {\em Math. Ann.}, 271:237--268, 1985.

\bibitem{kst2}
Emmanuel Kowalski, Abhishek Saha, and Jacob Tsimerman.
\newblock Local spectral equidistribution for {S}iegel modular forms and
  applications.
\newblock {\em Compositio Math.}, 148(2):335--384, 2012.

\bibitem{yifengliu}
Yifeng Liu.
\newblock Refined global {Gan–-Gross-–P}rasad conjecture for {B}essel periods.
\newblock {\em J. {R}eine {A}ngew. Math}, to appear.

\bibitem{NPS}
Hiro-aki Narita, Ameya Pitale, and Ralf Schmidt.
\newblock Irreducibility criteria for local and global representations.
\newblock {\em Proc. Amer. Math. Soc.}, 141(1):55--63, 2013.

\bibitem{nelson-pitale-saha}
Paul Nelson, Ameya Pitale, and Abhishek Saha.
\newblock Bounds for {R}ankin-{S}elberg integrals and quantum unique ergodicity
  for powerful levels.
\newblock {\em J. Amer. Math. Soc.}, 27(1):147--191, 2014.

\bibitem{Neu}
J{\"u}rgen Neukirch.
\newblock {\em Algebraic number theory}, volume 322 of {\em Grundlehren der
  Mathematischen Wissenschaften [Fundamental Principles of Mathematical
  Sciences]}.
\newblock Springer-Verlag, Berlin, 1999.
\newblock Translated from the 1992 German original and with a note by Norbert
  Schappacher, With a foreword by G. Harder.

\bibitem{PSS14}
Ameya Pitale, Abhishek Saha, and Ralf Schmidt.
\newblock {L}owest weight modules of {${\rm Sp}_4(\mathbb{R})$} and nearly
  holomorphic {S}iegel modular forms \verb"(expanded version)".
\newblock arXiv:1501.00524.

\bibitem{transfer}
Ameya Pitale, Abhishek Saha, and Ralf Schmidt.
\newblock Transfer of {S}iegel cusp forms of degree 2.
\newblock {\em Mem. Amer. Math. Soc.}, 232(1090):vi+107, 2014.

\bibitem{pitschram}
Ameya Pitale and Ralf Schmidt.
\newblock Ramanujan-type results for {S}iegel cusp forms of degree 2.
\newblock {\em J. Ramanujan Math. Soc.}, 24(1):87--111, 2009.

\bibitem{PS1}
Ameya Pitale and Ralf Schmidt.
\newblock Bessel models for {$\rm GSp(4)$}: {S}iegel vectors of square-free
  level.
\newblock {\em J. Number Theory}, 136:134--164, 2014.

\bibitem{PrasadTaklooBighash2011}
Dipendra Prasad and Ramin Takloo-Bighash.
\newblock Bessel models for {GS}p(4).
\newblock {\em J. Reine Angew. Math.}, 655:189--243, 2011.

\bibitem{qiu}
Yannan Qiu.
\newblock The {B}essel period functional on {$\SO(5)$}: the nontempered case.
\newblock {\em Preprint}, 2014.

\bibitem{res-sald}
Howard Resnikoff and R.~L. Saldana.
\newblock Some properties of {F}ourier coefficients of {E}isenstein series of
  degree two.
\newblock {\em J. Reine Angew. Math.}, 265:90--109, 1974.

\bibitem{NF}
Brooks Roberts and Ralf Schmidt.
\newblock {\em Local newforms for {GS}p(4)}, volume 1918 of {\em Lecture Notes
  in Mathematics}.
\newblock Springer, Berlin, 2007.

\bibitem{RobertsSchmidt2014}
Brooks Roberts and Ralf Schmidt.
\newblock Some results on {B}essel functionals for {${\rm GSp}(4)$}.
\newblock {\em Doc. Math.}, 21:467--553, 2016.

\bibitem{sahapet}
Abhishek Saha.
\newblock On ratios of {P}etersson norms for {Y}oshida lifts.
\newblock {\em Forum Math.}, 27(4):2361--2412, 2015.

\bibitem{sahaschmidt}
Abhishek Saha and Ralf Schmidt.
\newblock {Y}oshida lifts and simultaneous non-vanishing of dihedral twists of
  modular ${L}$-functions.
\newblock {\em J. London Math. Soc.}, 88(1):251--270, 2013.

\bibitem{sch}
Ralf Schmidt.
\newblock Iwahori-spherical representations of {${\rm GSp}(4)$} and {S}iegel
  modular forms of degree 2 with square-free level.
\newblock {\em J. Math. Soc. Japan}, 57(1):259--293, 2005.

\bibitem{schsaito}
Ralf Schmidt.
\newblock On classical {S}aito-{K}urokawa liftings.
\newblock {\em J. Reine Angew. Math.}, 604:211--236, 2007.

\bibitem{shimura75}
Goro Shimura.
\newblock On the {F}ourier coefficients of modular forms of several variables.
\newblock {\em Nachr. Akad. Wiss. G\"ottingen Math.-Phys. Kl. II}, 17:261--268,
  1975.

\bibitem{Waldspurger1985}
Jean-Loup Waldspurger.
\newblock Sur les valeurs de certaines fonctions {$L$} automorphes en leur
  centre de sym\'etrie.
\newblock {\em Compositio Math.}, 54(2):173--242, 1985.

\bibitem{watson-2008}
Thomas Watson.
\newblock Rankin triple products and quantum chaos.
\newblock 2008.
\newblock arxiv:0810.0425.

\bibitem{weissram}
Rainer Weissauer.
\newblock {\em Endoscopy for {${\rm GSp}(4)$} and the cohomology of {S}iegel
  modular threefolds}, volume 1968 of {\em Lecture Notes in Mathematics}.
\newblock Springer-Verlag, Berlin, 2009.

\end{thebibliography}
\bibliographystyle{plain}

\end{document}